\documentclass[11pt,reqno]{amsart}
\usepackage{amsmath,amsthm,amsfonts}
\usepackage {latexsym}
\usepackage{amssymb}

\usepackage{graphicx}

 \renewcommand{\thesection}{\arabic{section}}

\usepackage[usenames,dvipsnames]{xcolor}

\usepackage[capitalise]{cleveref}

\crefformat{equation}{(#2#1#3)}

\newtheorem{theorem}{Theorem}[section]

\newtheorem{thmg}{Theorem}

\newtheorem{thmh}{Theorem}

\newtheorem{thmha}{Theorem}

\newtheorem{propAlpha}{Proposition}

\newtheorem{propAlphaM}{Proposition}

\newtheorem{lemma}[theorem]{Lemma}

\newtheorem{proposition}[theorem]{Proposition}

\newtheorem{corollary}[theorem]{Corollary}

\newtheorem{remark}[theorem]{Remark}
\theoremstyle{remark}
\newtheorem*{remark*}{Remark}

\renewcommand{\theequation}{\thesection .\arabic{equation}}

\let\sect\section
\renewcommand\section{\setcounter{equation}{0}
\gdef\theequation{\thesection .\arabic{equation}}\sect}

\newcommand{\norm}[1]{\left\|#1\right\|}
\renewcommand{\epsilon}{\varepsilon}
\newcommand{\T}{\mathbb{T}}
\newcommand{\R}{\mathbb{R}}
\newcommand{\Z}{\mathbb{Z}}
\newcommand{\C}{\mathbb{C}}

\newcommand{\cB}{{\mathcal{B}}}

\newcommand{\cE}{{\mathcal{E}}}

\newcommand{\cS}{{\mathcal{S}}}

\newcommand{\IR}{{\mathbb{R}}}
\newcommand{\TT}{{\mathbb{T}}}

\newcommand{\tor}{\TT}

\newcommand{\IZ}{{\mathbb{Z}}}

\newcommand{\be}{\begin{eqnarray}}
\newcommand{\ee}{\end{eqnarray}}

\newcommand{\diam}{\mathop{\rm{diam}}}

\newcommand{\dist}{\mathop{\rm{dist}}}

\newcommand{\mes}{\mathop{\rm{mes}\, }}
\newcommand{\compl}{\mathop{\rm{compl}}}

\newcommand{\spec}{\mathop{\rm{spec}}}

\newcommand{\ve}{{\varepsilon}}

\newcommand{\nn}{\nonumber}

\def\beeq{\begin{equation}}
\def\eneq{\end{equation}}

\def\les{\lesssim}

\def\cS{{\mathcal S}}
\def\bm{\begin{matrix}}
\def\endm{\end{matrix}}

\def\Re{{\rm Re}}

\newcommand{\EQ}[1]{\begin{equation}\begin{split} #1 \end{split}\end{equation}}

\begin{document}

\title{Homogeneity of the Spectrum for Quasi-Periodic Schr\"odinger Operators}

\author{David Damanik, Michael Goldstein,
Wilhelm Schlag, and Mircea Voda.}

\address{Department of Mathematics, Rice University, 6100 S. Main St. Houston TX 77005-1892, U.S.A.}
\email{damanik@rice.edu}

\address{Department of Mathematics, University of Toronto, Toronto, Ontario, Canada M5S 1A1}
\email{gold@math.toronto.edu}

\address{Department of Mathematics, The University of Chicago, 5734 S. University Ave., Chicago, IL 60637, U.S.A.}
\email{schlag@math.uchicago.edu}

\address{Department of Mathematics, The University of Chicago, 5734 S. University Ave., Chicago, IL 60637, U.S.A.}
\email{mvoda@uchicago.edu}

\begin{abstract}
We consider the one-dimensional discrete Schr\"odinger operator
$$
\bigl[H(x,\omega)\varphi\bigr](n)\equiv -\varphi(n-1)-\varphi(n+1) + V(x + n\omega)\varphi(n)\ ,
$$
$n \in \IZ$, $x,\omega \in [0, 1]$ with real-analytic potential $V(x)$. Assume $L(E,\omega)>0$ for all $E$. Let $\mathcal{S}_\omega$ be the spectrum of $H(x,\omega)$. For all $\omega$ obeying the Diophantine condition $\omega \in \mathbb{T}_{c,a}$, we show the following: if $\mathcal{S}_\omega \cap (E',E'')\neq \emptyset$, then $\mathcal{S}_\omega \cap (E',E'')$ is homogeneous in the sense of Carleson (see \cite{Car83}). Furthermore, we prove, that if $G_i$, $i=1,2$ are two gaps with $1 > |G_1| \ge |G_2|$, then $|G_2|\lesssim \exp\left(-(\log \dist (G_1,G_2))^A\right)$, $A\gg 1$. Moreover, the same estimates hold for the gaps in the spectrum on a finite interval, that is, for $\cS_{N,\omega}:=\cup_{x\in\T}\spec H_{[-N,N]}(x,\omega) $, $N \ge 1 $, where $H_{[-N, N]}(x, \omega)$ is the Schr\"odinger operator restricted to the interval $[-N,N]$ with Dirichlet boundary conditions.
In particular, all these results hold for the almost Mathieu operator with $|\lambda| \neq 1$. For the supercritical almost Mathieu operator, we combine the methods of \cite{GolSch08} with Jitomirskaya's approach from \cite{Jit99} to establish most of the results from \cite{GolSch08} with $\omega$ obeying a strong Diophantine condition.
\end{abstract}

\maketitle

\tableofcontents

\section{Introduction}\label{sec:intro}

We consider quasi-periodic Schr\"odinger equations
\begin{equation}\label{eq:1.Sch}
\bigl[H(x,\omega)\varphi\bigr](n) \equiv - \varphi(n-1) - \varphi(n+1) +  V(x + n\omega)\varphi(n) = E\varphi(n)
\end{equation}
in the regime of positive Lyapunov exponents.  We assume that $V(x)$ is a $1$-periodic, real-analytic function. Recall that for irrational $\omega$, the spectrum of $H(x,\omega)$ does not depend on $x$. We denote it by $\mathcal{S}_\omega$. It was shown in \cite{GolSch11} that $\mathcal{S}_\omega$
is a Cantor set for almost every irrational $\omega$, in the regime of positive Lyapunov exponent. The main objective of this work is to show that the structure of the gaps is ``regular.'' More specifically, {\em a closed set $\mathcal{S}\subset\mathbb{R}$ is called homogeneous if there is $\tau > 0$ such that for any $E\in \mathcal{S}$ and any $0<\sigma\le \diam(\cS)$, the estimate
\begin{equation} \label{eq:1homogeneous}
|\mathcal{S}\cap(E-\sigma,E+\sigma)| > \tau\sigma
\end{equation}
holds} (see \cite{Car83}). We say also that $\mathcal{S}$ is $\tau$-homogeneous.

\begin{thmh}\label{thm:homogeneity}
Let
\begin{equation*}
\omega\in\T_{c,a}:= \left\{ \omega: \norm{n\omega}\ge \frac{c}{n(\log n)^a}, n\ge 1 \right\}.
\end{equation*}	
If $ E_0\in \cS_\omega $ and $ L(\omega,E_0)\ge \gamma>0 $ then there exists $ \sigma_0=\sigma_0(V,c,a,\gamma) $ such that
\begin{equation}\label{eq:local-homogeneity}
	|\cS_\omega\cap(E_0-\sigma,E_0+\sigma)|\ge \sigma/2,
\end{equation}
for all $ \sigma\in(0,\sigma_0] $. In particular:
\begin{enumerate}
	\item[(a)] If $ L(\omega,E)\ge \gamma $ for all $ E\in\R  $, then $ \cS_\omega $ is $\tau$-homogeneous with some
	$\tau=\tau(V,c,a,\gamma)$.
	
	\item[(b)] If $ L(\omega,E)\ge \gamma $ for all $ E\in(E',E'') $ and there exists $\epsilon > 0$ such that
	\begin{equation}\label{eq:maximality}
		\cS_\omega\cap(E'-\epsilon,E''+\epsilon)=\cS_\omega\cap(E',E''),
	\end{equation}
	then $\mathcal{S}_\omega\cap (E', E'')$ is either empty or $\tau$-homogeneous with some $\tau=\tau(V,c,a,\gamma,\epsilon)$.
\end{enumerate}
The previous statements also hold with $\cS_{N,\omega} := \cup_{x\in\T} \spec H_{[-N,N]}(x,\omega)$, $N \ge 1$, instead of
$\cS_\omega$ $($here $H_{[-N, N]}(x, \omega)$ is the Schr\"odinger operator restricted to the interval $[-N,N]$ with Dirichlet
boundary conditions$)$.
\end{thmh}

\noindent\textbf{Remarks.} (1)
If we introduce a coupling constant, that is, if we replace $V$ by $\lambda V$, we know by Sorets-Spencer \cite{SorSpe91} that
part (a) of our theorem applies for $ \lambda\ge\lambda_0(V) $. For part (b) we note that for energies near the edges of the interval
$ (E',E'') $ we don't know how much of the nearby spectrum afforded by \cref{eq:local-homogeneity} sits inside $ (E',E'') $. We
deal with this issue by imposing condition \cref{eq:maximality} which forces all the spectrum near $ (E',E'') $ to be in
$ (E',E'') $. In general, assuming that the Lyapunov exponent does not vanish throughout the spectrum,
the existence of the intervals $ (E',E'') $, to which part (b) of our theorem applies, follows from the continuity of the Lyapunov
exponent (see \cite{GolSch01},\cite{BJ}) and the density of gaps given by \cite{GolSch11}. Indeed, given $ E_0\in \cS_\omega $ such that
$ L(\omega,E_0)\ge \gamma>0 $, we can find an interval $ (E',E'') $ such that $ E_0\in(E',E'') $, $ L(\omega,E)\ge \gamma/2>0 $ for
$ E\in(E',E'') $, and $ E',E''\notin \cS_\omega $. The last condition insures that we have \cref{eq:maximality} with
$ \epsilon=\epsilon(\dist(\{ E',E'' \},\cS_\omega)) $.

(2) In general our theorem doesn't guarantee that $ \cS_\omega\cap \{ E\in \R:L(\omega,E)>0  \} $ is homogeneous
(unless we are in the setting of part (a)).
However, we do get that this is true for typical analytic potentials, in the sense of the Main Theorem of \cite{Avi13}.
Recall that in \cite{Avi13} it is shown that for typical analytic potentials there exist finitely many disjoint closed intervals
$ I_k $ such that $ \cS_\omega\subset \cup_k I_k $ and $ \cS_\omega\cap I_k $ is either absolutely continuous or pure point.
Furthermore, one has spectral uniformity in both subcritical and supercritical regimes. For the supercritical regime this means
that there exists $ \gamma>0 $ such that $ \cS_\omega\cap \{ E\in \R:L(\omega,E)>0  \}=\cS_\omega\cap \{ E\in \R:L(\omega,E)\ge\gamma  \} $.
One can now apply part (b)
of our theorem on each non-empty interval $ I_k\cap \{ E\in \R:L(\omega,E)\ge\gamma  \} $ to yield the homogeneity of the
spectrum in the supercritical regime.


(3) The strong Diophantine condition on $\omega$ can be relaxed. This improvement is one of the results in the ongoing work of Tao and Voda \cite{TaoVod15}. In the current work we use the existing results developed assuming the strong Diophantine condition in \cite{GolSch08} and \cite{GolSch11}.

\medskip
The homogeneity property of the spectrum of quasi-periodic Schr\"{o}dinger operators in the regime of {\em small coupling} was recently established in the paper~\cite{DamGolLuk14}. They consider the continuum quasi-periodic Schr\"{o}dinger operator
\begin{equation} \label{eq:1-1}
-\psi''(x) +  V(x) \psi(x) = E \psi(x), \qquad x \in \IR,
\end{equation}
where
\begin{equation} \label{eq:1.UU}
V(x) = \sum_{n \in \mathbb{Z}^\nu}\, c(n) e (x n \omega),
\end{equation}
\begin{equation}\label{eq:1-fouriercoeff}
|c(m)| \le \ve \exp(-\kappa_0 |m|)
\end{equation}
with $\kappa_0>0$, $\ve$ being small and with a Diophantine vector $\omega$,
\begin{equation}\label{eq:1PAI7-5-85a}
|n \omega| \ge a_0 |n|^{-b_0}, \quad n \in \mathbb{Z}^\nu \setminus \{ 0 \}
\end{equation}
for some
$$
0 < a_0 < 1,\quad \nu < b_0 < \infty.
$$
Using the estimates from~\cite{DamGol14} they establish the following relation between the gaps and the bands of the operator:

There exists $\varepsilon_0 = \varepsilon_0(\kappa_0, a_0, b_0) > 0$ such that for $0 < \varepsilon < \varepsilon_0$, the gaps in the spectrum of the operator $H$ can be labeled as $G_m = (E_m^-, E_m^+)$, $m \in \mathbb{Z}^\nu \setminus \{ 0 \}$, $G_0 = (-\infty, \underline{E})$ so that the following estimates hold:
\begin{enumerate}

\item[(i)] For every $m \in \mathbb{Z}^\nu \setminus \{ 0 \}$, we have
$$
E^+_m - E^-_m \le 2 \varepsilon \exp \Big( -\frac{\kappa_0}{2} |m| \Big).
$$

\item[(ii)] For every $m, m' \in \mathbb{Z}^\nu \setminus \{ 0 \}$ with $m' \neq m$ and $|m'| \ge |m|$, we have
$$
\dist ([E_m^-, E_m^+], [E_{m'}^-, E_{m'}^+]) \ge a |m'|^{-b},
$$
where $a, b > 0$ are constants depending on $a_0, b_0, \kappa_0, \nu$.

\item[(iii)] For every $m \in \mathbb{Z}^\nu \setminus \{ 0 \}$,
$$
E_m^--\underline{E} \ge a |m|^{-b},
$$
\end{enumerate}

This feature was not known for the almost Mathieu operator even in the regime of small coupling. The homogeneity property can be derived from
(i)--(iii). In the current paper we establish a slightly weaker version of (i)--(iii).

\begin{thmg}\label{thm:gaps-separation}
Let $ \omega\in\T_{c,a} $ and assume $ L(\omega,E)\ge \gamma>0 $ for any $ E\in(E',E'') $. There exists $ N_0(V,c,a,\gamma) $ such that if $ N\ge N_0 $ and $ G_1 $, $ G_2 $ are two gaps in $ \cS_\omega\cap(E',E'') $	with $ |G_1|,|G_2|>\exp(-N^{1-}) $ then $ \dist(G_1,G_2)>\exp(-(\log N)^{C_0}) $, with $C_0=C_0(V,c,a,\gamma) $. The same statement holds for gaps in $ \cS_{\bar N,\omega}\cap(E',E'') $ with $ \bar N\ge N $.	
\end{thmg}

\smallskip
It is natural to inquire about the precise calibration between the gaps and the bands. In particular, is it true that, in \cref{thm:gaps-separation}, one has
$$
\dist (G_1,G_2) \ge a |N|^{-b},
$$
with $a, b > 0$ being constants depending on $V$, $\omega$ and the lower bound  $L(E)\ge\gamma>0$? Moreover, if so, are these optimal estimates?

\bigskip
Consider the almost Mathieu operator
\begin{equation} \label{eq:3.amopera}
[H(x,\omega) \phi](n) = -\phi(n-1)-\phi(n+1)+2\lambda\cos(2\pi(x+n\omega)) \phi(n),\quad n\in \mathbb{Z}.
\end{equation}

It is a fundamental fact that the Lyapunov exponent here obeys
\[
L(\omega,E) \ge \log |\lambda|
\]
for all $E$. Thus, as a particular case of \cref{thm:homogeneity} and as a consequence of Aubry duality, we have the following.

\begin{thmha}\label{thm:homogeneity-AMO}
Let $|\lambda|\neq 1$ and $\omega\in \mathbb{T}_{c,a}$. The set $\mathcal{S}_\omega$ is a $\tau$-homogeneous set for some $\tau=\tau(c,a,\lambda)$. Furthermore, the estimates in \cref{thm:gaps-separation} hold.
\end{thmha}

The relevance of the homogeneity property to the inverse spectral theory of almost-periodic potentials (or Jacobi matrices with almost periodic coefficients) was established in the remarkable work by Sodin and Yuditskii \cite{SodYud95, SodYud97}. They studied the inverse spectral problem for reflectionless Jacobi matrices whose spectrum is a given homogeneous set. The reflectionless potentials were introduced, in the continuum setting, by Craig \cite{Cra89}. Reflectionless potentials are very relevant to the spectral theory of ergodic potentials. Different classes of potentials, which are in fact reflectionless, were studied, prior to \cite{Cra89}, in the basic works on ergodic potentials by Deift and Simon \cite{DeiSim83}, Johnson \cite{Joh82}, Johnson and Moser \cite{JohMos82}, Kotani  ~\cite{Kot84},\cite{Kot87}. It was shown in the work \cite{Cra89} that being reflectionless is the key feature which allows for the development of a number of fundamental objects from the periodic theory like auxiliary spectrum, trace formula,
product expansions; see also the work by Gesztezy and Simon ~\cite{GesSim96}. Employing the version of the trace formula from ~\cite{GesSim96}, Gesztezy and Yuditskii \cite{GesYud06} found another remarkable consequence of the homogeneity property combined with being reflectionless: the spectrum is purely absolutely continuous. See also the paper by Poltoratski and Remling~\cite{PolRem09}, where an even stronger result was established.

\smallskip

In view of these results and the results of the current paper it seems very natural to investigate the connection between the homogeneity property and the spectral phase transition theory of quasi-periodic potentials, see the work by Avila \cite{Avi13}. Namely, we would like to pose the following question:

\medskip

\underline{Problem 1.} Consider
\begin{equation}\label{eq:1.SchP2}
\bigl[H(x,\omega)\varphi\bigr](n) \equiv - \varphi(n-1) - \varphi(n+1) +  \lambda V(x + n\omega)\varphi(n) = E\varphi(n)
\end{equation}
with real analytic $V$ and Diophantine $\omega$. It is known that for small $\lambda$, the operator has a complete set of Bloch-Floquet eigenfunctions. We expect that, in analogy to the small-coupling result in the continuum case from \cite{DamGolLuk14}, one can prove also that the spectrum is homogeneous and, moreover, that the calibration estimates (i)--(iii) for gaps and bands hold. Assume that the Lyapunov exponent $L(E,\lambda)$ vanishes on the spectrum for all $0<\lambda<\lambda_0$. Can one find a complete set of Bloch-Floquet eigenfunctions for $0<\lambda<\lambda_0$? The main issue here is how to control the homogeneity property of the spectrum using the zero Lyapunov exponent on the spectrum. Indeed, while vanishing Lyapunov exponents on a set of positive Lebesgue measure imply the presence of absolutely continuous spectrum, the homogeneity of the spectrum is a sufficient condition for \emph{purely} absolutely continuous spectrum. Once the latter property has been established, the existence of a complete set of Bloch-Floquet eigenfunctions follows from the work of Kotani \cite{Kot84}, Deift-Simon \cite{DeiSim83}, and Avila-Krikorian \cite{AK}.

\medskip

Finally, we want to stress the fact that the analysis of ``fine properties" of the spectral set $ \cS_{N,\omega}:=\cup_{x\in\T}\spec H_{[-N,N]}(x,\omega)$ on \textbf{a finite interval}, especially with general analytic $V$, seems to be a very interesting problem in its own right. The numerical plots of the eigenvalues $E^{(N)}_j(x)$ of $H_{[-N,N]}(x,\omega)$ $($Rellich parametrization$)$ look very complicated, see the following figures.

\begin{figure}[htb]
\includegraphics{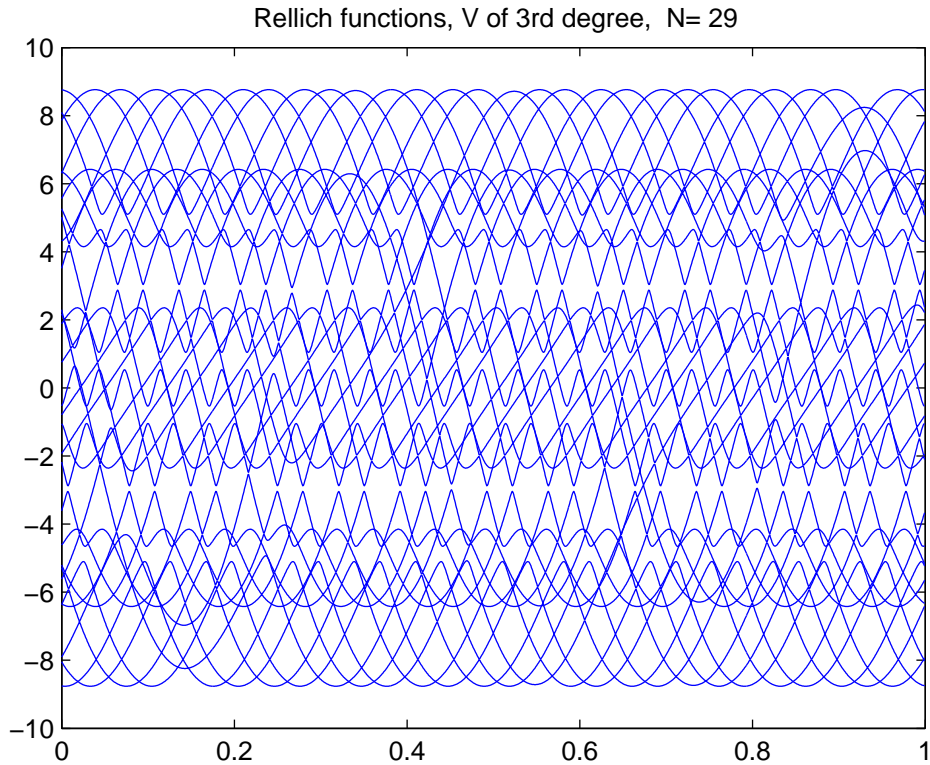}
\caption{\label{fig:4}Rellich functions.}
\end{figure}
One can see some ``almost gaps" shadowed by ``rare graphs fragments", see for instance Figure~\ref{fig:4} between the spectral value levels
$E=2$ and $E=4$. Even for the almost Mathieu case, the picture still has some ``gaps shadowing", see Figure~\ref{fig:5}.
\begin{figure}[htb]
\includegraphics{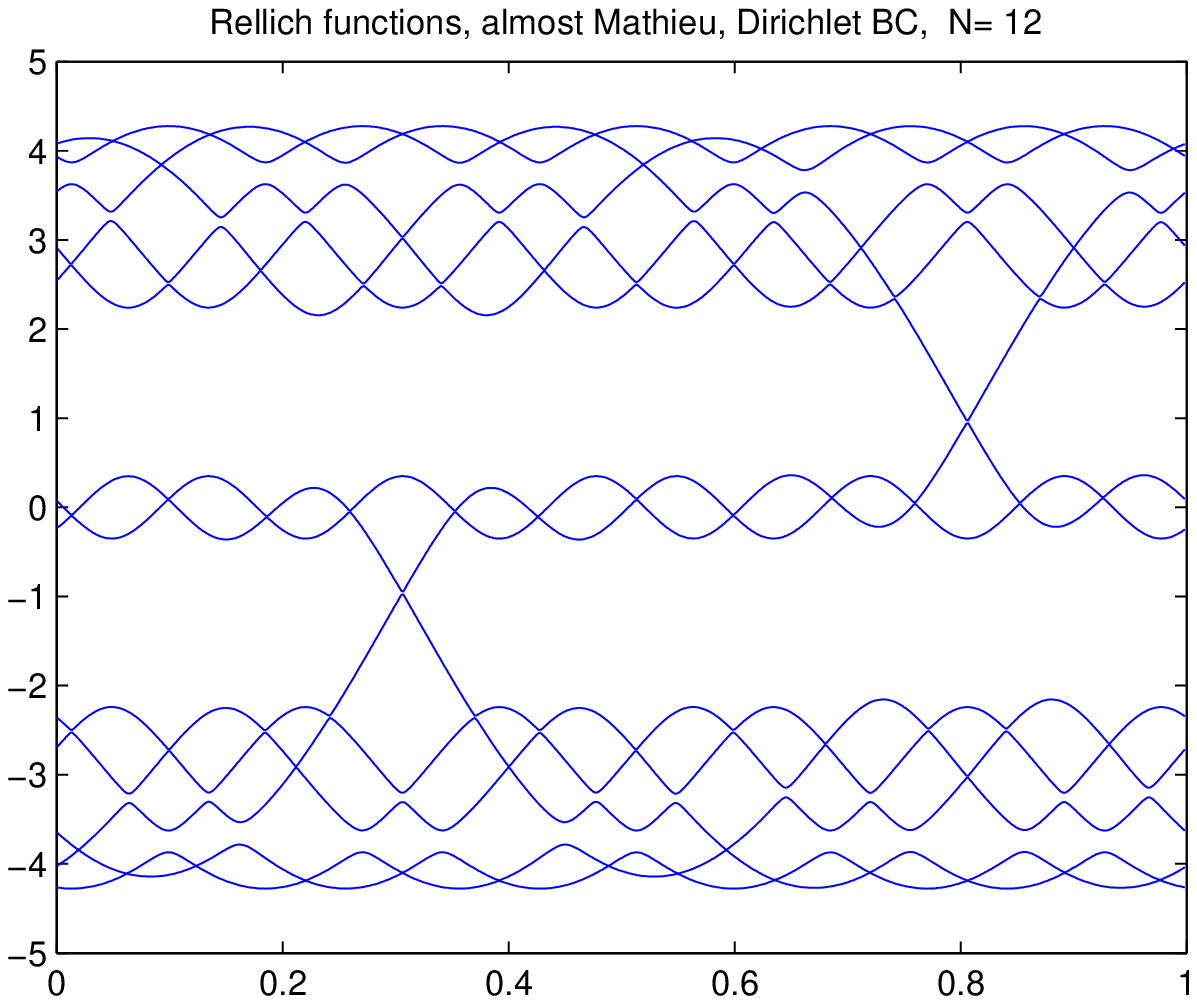}
\caption{\label{fig:5} Rellich functions.}
\end{figure}
\begin{figure}[htb]
\includegraphics{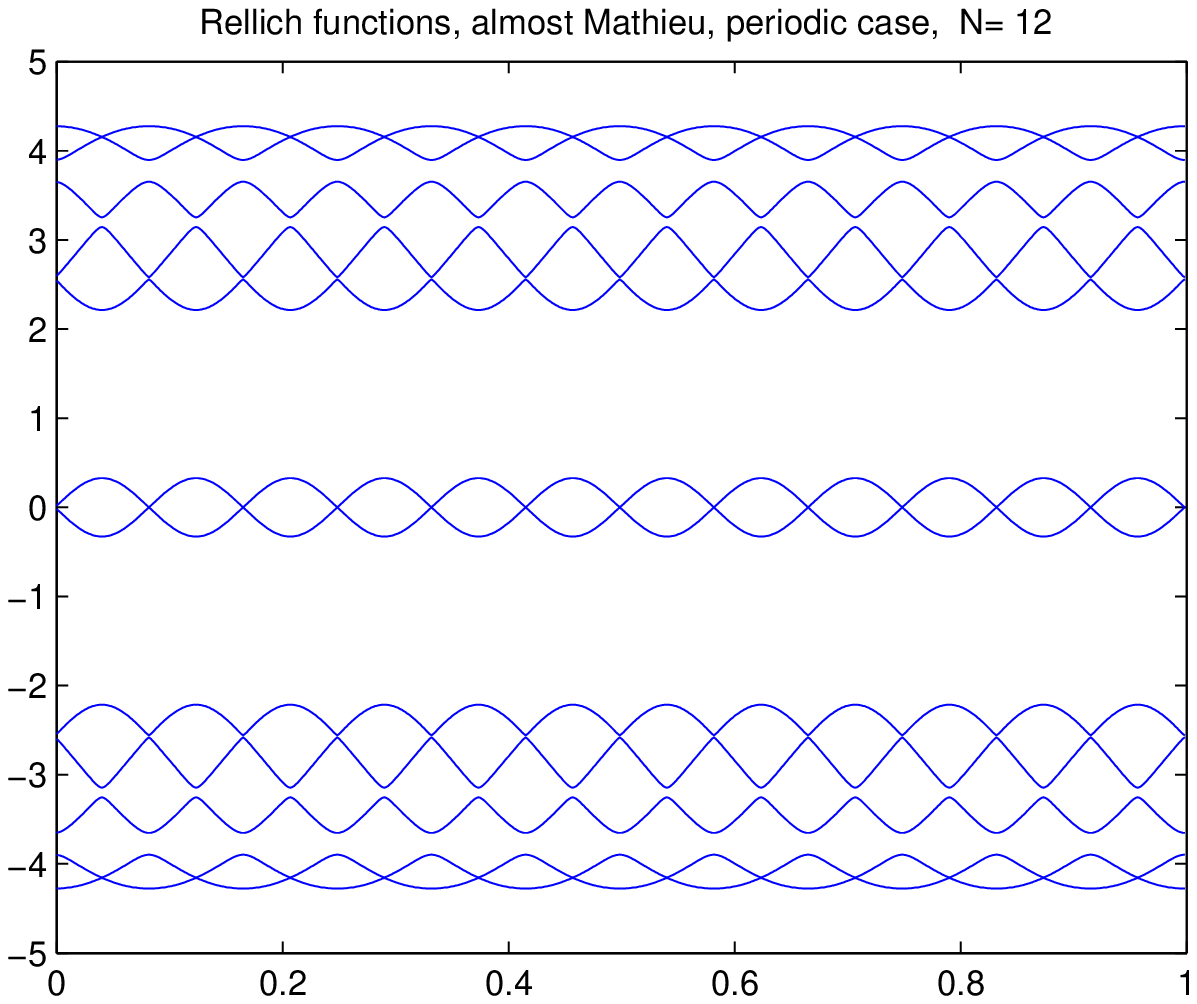}
\caption{\label{fig:6}Rellich functions.}
\end{figure}
The picture simplifies under periodic boundary conditions, see Figure~\ref{fig:6}.

\bigskip

Finally, we would like to pose the following problem:

\medskip

\underline{Problem 2.} (a) Describe as accurately as possible the ``true" gaps in the spectral set $ \cS_{N,D}:=\cup_{x\in\T}\spec H_{[-N,N],D} (x, \omega)$ on a finite interval with Dirichlet boundary conditions. In particular, determine the gaps of smallest size.

(b) Develop a description for the spectral set $ \cS_{N,P}:=\cup_{x\in\T}\spec H_{[-N,N],P}(x,\omega)$ on a finite interval with periodic boundary conditions.

(c) Develop a description of the spectral sets on a finite interval for rational approximations of the frequency $\omega$.

\subsection{Description of the Method}

As mentioned before, the homogeneity of the spectrum for continuous Schr\"odinger operators and small coupling constant was recently established in \cite{DamGolLuk14} via detailed quantitative results concerning the structure of the gaps in the spectrum. We show in this paper that in the regime of positive Lyapunov exponent, homogeneity can be obtained with less machinery. In fact one does not even need to use finite scale localization. Rather, we use finite scale approximate eigenvalues rather than eigenfunctions. This approach only relies on the availability of a large deviation estimate; compare \cite{Bou02}, where a similar idea was used.

This method has the advantage of avoiding the removal of a ``non-arithmetic'' set of frequencies, which would be needed to eliminate the double resonances, as required in order to establish localization. To capture the infinite volume spectrum, we establish a criterion for given $E_0$ to fall in the spectrum on the whole lattice, see Lemma~\ref{lem:spectrum-criterion}.

The most basic mechanism behind the homogeneity of the spectrum is the {\em Wegner estimate}. It is a finite volume version of the fact that the {\em integrated density of states is H\"{o}lder continuous}. For any $x , \omega \in
\tor$, let
\EQ{\label{eq:Ej}
\bigl\{E_j^{(N)} (x, \omega)\bigr\}_{j=1}^N,\quad
\bigl\{\psi_j^{(N)} (x, \omega, \cdot)\bigr\}_{j=1}^N
}
denote the eigenvalues and a choice of normalized eigenvectors of $H_{[1, N]}(x, \omega)$, respectively.  The Wegner estimate amounts to the fact that the graphs of $E_j^{(N)} (x, \omega)$ cannot be ``too flat''. See the discussion of the quantitative version of this issue in Remark~\ref{rem:Wegner}.
The other main reason for  the homogeneity of the spectrum is the fragmentary stabilization of the graphs of the Dirichlet eigenvalues plotted against the phases at different scales, see Figure~\ref{fig:stability}. This allows for good control on the structure of the spectrum on the whole lattice $\mathbb{Z}$ via the spectrum on intervals $[-N,N]$ with large $N$. Thus, the Wegner estimate makes it possible to obtain finite scale spectral segments of considerable size that we can then screen, via fragmentary stabilization, to obtain relatively large sets in the infinite volume spectrum.
Heuristically, this is how the proof of \cref{thm:homogeneity} proceeds.

\begin{figure}[htb]
\includegraphics{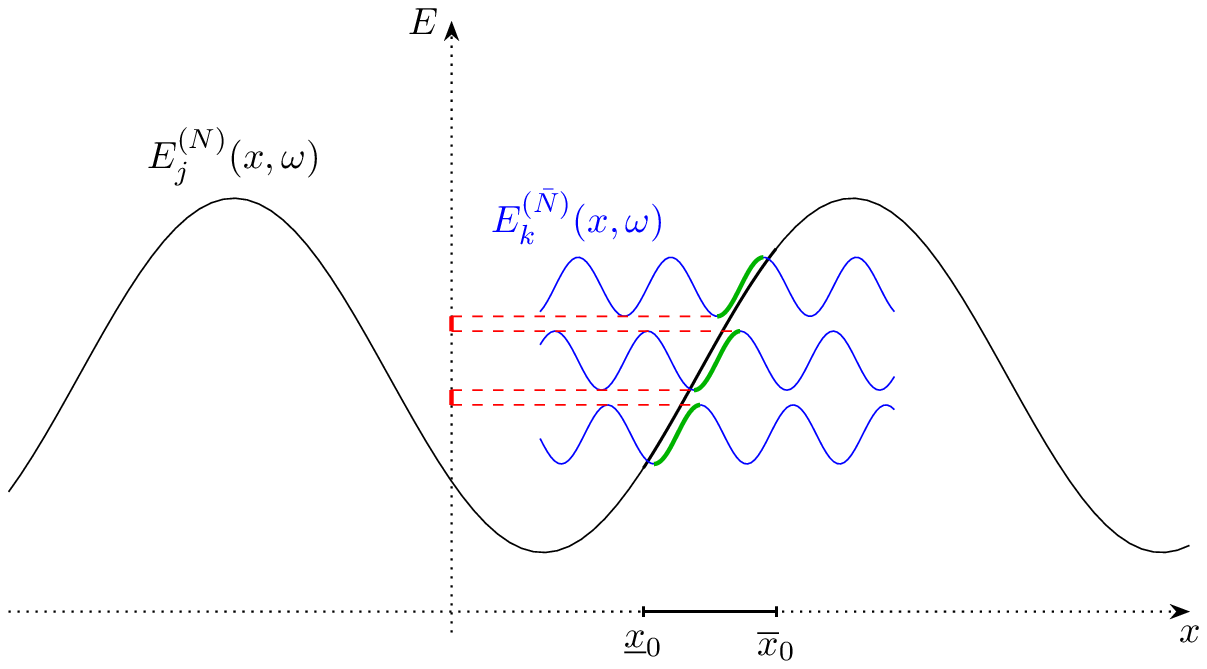}
\caption{\emph{\label{fig:stability} Fragmentary stability of spectral segments}: The green fragments ($ \bar N\gg N $) are very close
to the spectral segment $ E_j^{(N)}(x,\omega) $, $ x\in[\underline x_0,\overline x_0] $. The small red segments on the
$ E $ axis represent the exceptional set.}
\end{figure}

As we already mentioned, the resolution of the fragmentary stability picture is accurate enough to allow for the proof of homogeneity to go through. However, it seems that for possible future refinements of the result, one needs the more detailed picture given by finite scale localization.  This is why, after we prove the main result, we also develop the finite scale localization approach. The novelty here is that we focus on the almost Mathieu
operator for which we can establish the results without removal of any frequencies $ \omega\in\T_{c,a} $. This is due to the method of Jitomirskaya \cite{Jit99} for eliminating resonances. The advantage of this method resides with the fact that it explicitly identifies the resonant phases as $x=m\omega/2$ mod $1$, $m\in \mathbb{Z}$, and there is no need to eliminate any further $\omega$'s.

We now give a rough outline of the main ingredients for both parts of the paper. Most of them were developed in \cite{GolSch08,GolSch11}. In fact, we shall cite several results from these papers as part of our argument, see Propositions~A--I below. The following three items describe basic properties of the transfer matrix formalism. Throughout, it is essential that the Lyapunov exponents are positive.

\begin{itemize}

\item Large deviation estimate for the characteristic determinants of the Dirichlet problem on a finite interval.

\item H\"{o}lder continuity of the Lyapunov exponent.

\item Uniform upper bounds for the Dirichlet characteristic determinants.

\end{itemize}

The next three items build upon these foundations and describe essential features of the spectral theory, in particular the localization of the eigenfunctions.

\begin{itemize}

\item A version of the Wegner estimate.

\item Elimination of double resonances on finite intervals.

\item Exponential localization on finite intervals.

\end{itemize}

Finally, these tools feed into the following facts, which lie deeper and are of crucial importance to understanding the structure of the gaps in the spectrum.

\begin{itemize}

\item Quantitative separation of the Dirichlet eigenvalues on a finite interval.

\item Formation of the spectrum on the whole lattice from the spectra on finite intervals.

\end{itemize}

This paper is not self-contained since we refer to reader to \cite{GolSch08,GolSch11} for some of the rather involved proofs of the aforementioned technical ingredients. These results are used repeatedly in this work. On the other hand, some of the results derived from them, such as the Wegner estimate, are easy to obtain and we present the proofs of these facts.

\smallskip

Certain finer spectral properties, most notably the localization of the eigenfunctions and separation of the eigenvalues in the setting of \cref{thm:homogeneity}, require elimination of a Hausdorff dimension zero set of $\omega$. So we cannot follow that route for the almost Mathieu operator. Therefore, approximately half of the work in this paper is devoted to establishing the needed ingredients for the Mathieu operator with $|\lambda|>1$ and with arbitrary $\omega\in \T_{c,a}$.

On first reading, it is recommended to focus entirely on the proofs of \cref{thm:gaps-separation} and Theorem~\ref{thm:homogeneity}.


\section{Transfer Matrices and the Wegner Estimate}

We start by recalling the basics of the transfer matrix formalism. If $ \psi $ is a solution of the difference equation $ H(x,\omega)\psi=E\psi $, then we have
\begin{equation*}
	\begin{bmatrix}
		\psi(b+1)\\ \psi(b)
	\end{bmatrix}
	=M_{[a,b]}(x,\omega,E)
	\begin{bmatrix}
		\psi(a)\\ \psi(a-1)
	\end{bmatrix},
\end{equation*}
where the transfer matrix is given by
\begin{equation*}
	M_{[a,b]}(x,\omega,E)= \prod_{k=b}^a \begin{bmatrix}
		V(x+k\omega)-E & -1 \\
		1 & 0
	\end{bmatrix}.
\end{equation*}
We let $ M_N=M_{[1,N]} $. The Lyapunov exponent is defined by
\begin{equation*}
	L(\omega,E)=\lim_{N\to\infty}\frac{1}{N}\int_0^1 \log\norm{M_N(x,\omega,E)}\,dx
	\stackrel{a.s.}{=} \inf_{N} \frac{1}{N}\log\norm{M_N(x,\omega,E)}.
\end{equation*}
The H\"older continuity of the Lyapunov exponent as a function of the energy was
established in \cite{GolSch01} under a strong Diophantine condition. The result was improved in \cite{YouZha14} to hold even for weak Liouville frequencies. The dependence of the H\"older exponent on $\gamma$ (see the following proposition) was removed (for strongly Diophantine frequencies) in Proposition 8.3 of~\cite{Bou05}.

\begin{propAlpha}\label{prop:Lyapunov-continuity}
  	Assume $\omega\in \T_{c,a}$ and $L(\omega,E_0)\ge \gamma>0$. There
    exists $\epsilon_0=\epsilon_0(V,c,a,\gamma)$ such that
    $L(\omega,E)\ge \gamma/2$ for any $E\in (E_0-\epsilon_0,E_0+\epsilon_0)$. Moreover, there exists
    $\alpha_0=\alpha_0(V,c,a)>0$ such that
    \[
    |L(\omega,E_1)-L(\omega,E_2)|\le C(V,c,a,\gamma)|E_1-E_2|^{\alpha_0},
    \]
    for any $E_j\in (E_0-\epsilon_0,E_0+\epsilon_0)$, $j=1,2$.
\end{propAlpha}

The above result is essentially \cite[Thm. 6.1]{GolSch01}. The first statement is implicit in \cite{GolSch01}, but it also follows
explicitly from \cite{BJ}.

Next we focus on results concerning the finite scale Dirichlet determinants. Let $H_{[a, b]}(x, \omega)$ be the Schr\"odinger operator defined via \eqref{eq:1.Sch} on a finite interval $[a,b]$ with Dirichlet boundary conditions, $\psi(a-1)=0, \psi(b+1)=0$. Let $f_{[a, b]}(x, \omega, E)=\det (H_{[a, b]}(x, \omega)-E)$ be its characteristic polynomial.
One has
\begin{equation}\label{eq:1.Dirdet'}
f_{[a, b]}(x, \omega, E) = f_{b-a+1} \bigl(x+(a-1)\omega, \omega, E\bigr),
\end{equation}
where
\begin{equation}\label{eq:Dirichlet-det}
    \begin{aligned}
    f_N(x,\omega, E)  &= \det\bigl(H_N(x, \omega)-E\bigr)\\
    & =
    \begin{vmatrix}
     V\bigl(x+\omega\bigr) - E & -1 & 0  &\cdots &\cdots & 0\\[5pt]
    -1 & V\bigl(x+2\omega\bigr) - E & -1 & 0 & \cdots & 0\\[5pt]
    \vdots & \vdots & \vdots & \vdots & \vdots & \vdots\\
    &&&&&-1 \\[5pt]
    0 & \dotfill & 0 & -1 && V\bigl(x+N\omega\bigr) - E
    \end{vmatrix}
    \end{aligned}
\end{equation}
It is known that we also have
\begin{equation}\label{eq:transfer-and-determinants}
	M_{[a,b]}(x,\omega,E)= \begin{bmatrix}
		f_{[a,b]}(x,\omega,E) & -f_{[a+1,b]}(x,\omega,E)\\
		f_{[a,b-1]}(x,\omega,E) & -f_{[a+1,b-1]}(x,\omega,E)
	\end{bmatrix}.
\end{equation}
It was shown in \cite{GolSch08} that through this relation it is possible to pass from large deviation estimates for the transfer
matrix to large deviation estimates for the determinants. The following large deviation estimate for the determinants is a basic tool
in our approach, see Corollary~3.6 in~\cite{GolSch08}.

\begin{propAlpha}\label{prop:ldt}
    Let $\omega\in \mathbb{T}_{c,a}$, $ E\in \R $ be such that
    $L(\omega,E)\ge\gamma > 0$. There exists $C_0=C_0(V,c,a,\gamma)$ such that
  	\begin{equation}\label{eq:2.ldtd}
		\mes \left \{ x\in \T: \left| \log|f_N(x,\omega,E)|-NL(\omega,E) \right|>H \right\}
			\le C\exp(-H/(\log N)^{C_0})
	\end{equation}
	for all $ H>(\log N)^{C_0} $ and $ N\ge 2 $.
	Moreover, the set on the left-hand side is
    contained in the union of $ \les N $ intervals, each of which does not
    exceed the bound stated in~\eqref{eq:2.ldtd} in measure.
\end{propAlpha}

Subharmonic functions can deviate only towards large negative values but not large positive ones.
This explains the following result which is implied by Proposition~4.3 in \cite{GolSch08}.

\begin{propAlpha}\label{prop:uniform-upper-bound}
	Let $\omega\in \mathbb{T}_{c,a}$, $ E\in \R $ be such that
    $L(\omega,E)\ge\gamma > 0$. There exist $C_0=
    C_0(V,c,a,\gamma)$ and $C=C(V)$ such that
    \begin{equation}\label{eq2.61}
    \sup\limits_{x\in\tor} \log | f_N (x,\omega,E)| \le
    NL(\omega,E)+ C(\log N)^{C_0}
    \end{equation}
    for any $N\ge 2$.
\end{propAlpha}

Wegner's estimate now follows easily.

\begin{propAlpha}\label{prop:Wegner}
    Let $\omega\in \mathbb{T}_{c,a}$, $ E\in \R $ be such that
    $L(E,\omega)\ge\gamma > 0$. There exists $C_0=C_0(V,c,a,\gamma)$ such that for $H > (\log N)^{2C_0}$ and
    $ N\ge 2 $, one has
    \begin{equation}\label{eq:2.wegn}
    	\mes \bigl\{x \in \tor\::\: \dist\bigl(\spec(H_N(x,\omega)), E\bigr) < \exp(-H)\bigr\}
    	\les \exp \bigl(-H/(\log N)^{C_0}\bigr).
    \end{equation}
    Moreover, the set on the left-hand side is
    contained in the union of  $\les N$ intervals, each of which does not
    exceed the bound stated in~\eqref{eq:2.wegn} in measure.
\end{propAlpha}
\begin{proof}
    By Cramer's rule,
    \begin{equation*}
       \bigl|\bigl(H_N(x, \omega) - E\bigr)^{-1} (k, m)\bigr| = {\big
    	|f_{[1, k-1]}\bigl(x, \omega, E\bigr)\big |\, \bigl| f_{[m+1, N]}
    	\bigl(x, \omega, E \bigr)| \over \big | f_N\bigl(x, \omega,E\bigr)\big |}.
    \end{equation*}
    By \cref{prop:uniform-upper-bound},
    \begin{equation*}
	    \nn \log \big | f_{[1, k-1]}\bigl(x), \omega, E\bigr)\big | + \log
    	\big |f_{[m+1, N]}\bigl(x, \omega, E\bigr)\big | \le NL(\omega,E)
	    + C(\log N)^{C_0}.
    \end{equation*}
    for any $x \in \tor$.  Therefore,
    \begin{equation*}
	     \big\| \bigl(H_N(x, \omega) - E\bigr)^{-1} \big\| \le N^2\
	    {\exp\bigl(NL(\omega,E) + C(\log N)^{C_0}\bigr)\over \big
    	|f_N\bigl(x, \omega, E\bigr)\big|}
    \end{equation*}
    for any $x \in \tor$.  Since
    $$
    	\dist\bigl(\spec\bigl(H_N(x,\omega)\bigr),E\bigr)
    	= \big\| \bigl(H_N(x,\omega) -E\bigr)^{-1}\big\|^{-1}\ ,
    $$
    the lemma follows from \cref{prop:ldt}.
\end{proof}

The following result is an immediate consequence of the Wegner estimate \eqref{eq:2.wegn} and the continuity of the functions $E_j^{(N)} (x, \omega)$.

\begin{corollary}\label{cor:size-of-spectral-segment}
	Let $\omega\in \mathbb{T}_{c,a}$ and assume
    $L(E,\omega)\ge\gamma > 0$ for $ E\in(E',E'') $. There exists $C_0=C_0(V,c,a,\gamma)$ such that
    for $H > (\log N)^{2C_0}$ and $ N\ge 2 $, one has that if $ I $ is an interval satisfying
    \begin{equation*}
    	|I|\ge \exp(-H/(\log N)^{C_0})\text{ and }E_j^{[-N,N]}(I,\omega)\subset(E',E''),
    \end{equation*}
    then
    \begin{equation*}
    	|E_j^{[-N,N]}(I,\omega)|\ge 2\exp(-H).
    \end{equation*}
\end{corollary}

In the following remark, $a\sim b$ for $a,b>0$ means that these number are comparable
up to fixed multiplicative constants (say, within a factor of~$2$). Moreover, $a\gg b$ means that $\frac{a}{b}\ge C$ for
some large constant $C$.

\begin{remark}\label{rem:Wegner}\upshape
The Wegner estimate is a fundamental tool which has been applied to the problem of localization of eigenfunctions in both the quasi-periodic and the random settings. For the problem under consideration here, namely the homogeneous nature of the spectrum, our reading of Wegner's estimate
is as follows. Let $E\in\mathbb{R}$ be arbitrary and recall the eigenvalues as defined in~\eqref{eq:Ej}. Assume that
\begin{equation}
\label{eq:2.wegn11}
|E- E_j^{(N)} (x_0, \omega)|\leq \exp(-(\log N)^A)
\end{equation}
for some $N$ and $x_0$ and $A\gg 1$. Then, with $\sigma$ calibrated against $N$ such that
\begin{equation} \label{eq:1homogeneous1234}
\sigma\sim \exp(-(\log N)^A),
\end{equation}
the intersection
\begin{equation} \label{eq:1homogeneous125}
(E-\sigma,E+\sigma)\cap
\{E_j^{(N)} (x, \omega): x\in (x_0-\exp(-(\log N)^B), x_0+\exp(-(\log N)^B))\},
\end{equation}
$B:=A/2$, contains an interval $\mathcal{I}_E$ with
\begin{equation}
\label{eq:2.wegn13}
\big|\mathcal{I}_E\big|\ge \mathbf{w}_N:=\exp(-(\log N)^A)\sim\sigma.
\end{equation}
Note that this is a special case of the previous corollary, using the largest possible values of~$H$.
\end{remark}

The next two results address the relation between the distance of an energy to the spectrum and the large deviation estimate from Proposition~\ref{prop:ldt}. Recall that for any $x_0,x$, one has
\begin{equation}\label{eq:Hx-vs-Hx0}
	|E_{j}^{(N)} (x,\omega)-E_{j}^{(N)} (x_0,\omega)|
	\le \norm{H_N(x,\omega)-H_N(x_0,\omega)}\le C(V)|x-x_0|.
\end{equation}

\begin{lemma}\label{lem:spec-to-ldt}
    Let $\omega\in \T_{c,a}$, $ E\in \R $ be such that $L(E,\omega)\ge\gamma > 0$ and let $ x\in \T $.
    There exist $ N_0(V,c,a,\gamma,E) $, $ C_0(V,c,a,\gamma,E) $ such that for any $ N\ge N_0 $, we have that
    if $$\dist(E,\spec(H_N(x,\omega))\ge\exp(-K),$$ $ K\gg 1 $, then
    \begin{equation*}
    	\log|f_N(x,\omega,E)|\ge NL(\omega,E)-K(\log N)^{C_0}.
    \end{equation*}
\end{lemma}

\begin{proof}
    Due to Proposition \ref{prop:ldt} there exists $ x' $ such that $ |x'-x|<\exp(-K\log N) $ and
    \begin{equation*}
    	\log|f_N(x',\omega,E)|>NL(\omega,E)-K(\log N)^C.
	\end{equation*}
    Due to \eqref{eq:Hx-vs-Hx0} and our assumption on $ E $, we obtain
    \begin{multline*}
    	|E_{j}^{(N)} (x',\omega)-E_{j}^{(N)} (x,\omega)||E_{j}^{(N)} (x,\omega)-E|^{-1}
     	\le C(V)|x'-x||E_{j}^{(N)} (x,\omega)-E|^{-1}\\
     	\le \exp(-(K\log N)/2)<1/2,
    \end{multline*}
    and therefore
    \begin{multline*}
    	\big|\log \big|E-E^{(N)}_j(x',\omega)\big|-\log \big|E-E^{(N)}_j(x,\omega)\big|\big|
     	\le 2|E_{j}^{(N)} (x',\omega)-E_{j}^{(N)} (x,\omega)||E_{j}^{(N)} (x,\omega)-E|^{-1}\\
     	\le 2\exp(-(K\log N)/2),	
    \end{multline*}
    \begin{equation*}
	     \big|\log \big |f_N\bigl(x',\omega,E\bigr)\big |-\log \big |f_N\bigl(x,\omega,E\bigr)\big |\big|
	     \le 2N\exp(-(K\log N)/2)<1.
    \end{equation*}
    This yields the desired conclusion.
\end{proof}

The usefulness of a lower bound on the determinant as in the previous lemma can be seen from the following result.

\begin{lemma}[{\cite[Lem. 6.1]{GolSch11}}]\label{lem:ldt-to-decay}
	Let $ \omega\in \T_{c,a} $, $ E\in\R $, $ L(\omega,E)>\gamma>0 $, and $ N\ge N_0(V,a,c,\gamma,E) $.
	Furthermore, assume that
	\begin{equation*}
		\log|f_N(x,\omega,E)|>NL(\omega,E)-K/2
	\end{equation*}
	for some $ x\in \T $, $ K>(\log N)^{C_0} $. Then
	\begin{equation*}
		\left| (H_N(x,\omega)-E)^{-1}(j,k) \right|
		\le \exp(-\gamma|j-k|+K),
	\end{equation*}
	\begin{equation*}
		\norm{(H_N(x,\omega)-E)^{-1}}\le \exp(K).
	\end{equation*}
	In particular we have $ \dist(E,\spec H_N(x,\omega))\ge \exp(-K) $.
\end{lemma}

\begin{proof}
Apply Cramer's rule as in the proof of Wegner's estimate.
\end{proof}

We will use the following immediate consequence of \cref{lem:spec-to-ldt,lem:ldt-to-decay}.

\begin{lemma}\label{lem:spec-to-decay}
	Let $ \omega\in\T_{c,a} $, $ E\in \R $ be such that $ L(\omega,E)\ge \gamma>0 $ and let $ x\in \T $. There exist
	$ N_0(V,c,a,\gamma,E) $ and $ C_0(V,c,a,\gamma,E) $ such that for any $ N\ge N_0 $, we have that if
	$$ \dist(E,\spec H_N(x,\omega))\ge \exp(-K), $$
	$ K\gg 1 $, then
	\begin{equation*}
		\left| (H_N(x,\omega)-E)^{-1}(j,k) \right|\le \exp(-\gamma|j-k|+2K(\log N)^{C_0}).
	\end{equation*}
\end{lemma}

It is natural to link eigenfunctions of the finite volume operators to (generalized) eigenfunctions in infinite volume. The standard tool for this is the {\em Poisson formula:} for any solution of the difference equation $ H(x,\omega)\psi=E\psi $, we have
\beeq \label{eq:Poisson}
	\psi(m) = (H_{[a,b]}-E)^{-1}(m, a)\psi(a-1) + (H_{[a,b]}-E)^{-1}(m,b+1)\psi(b+1),\quad m \in [a, b].
\eneq
This identity was introduced into the theory of localized eigenfunctions in the fundamental work on the Anderson model by Fr\"ohlich and Spencer \cite{FroSpe83}. The Poisson formula tells us that the {\em decay of the Green function implies the decay of the eigenfunction wherever the Green function exists.} Lemmas~\ref{lem:ldt-to-decay}, \ref{lem:spec-to-decay} demonstrate how to effectively apply the Poisson formula in the regime of positive Lyapunov exponents, by being able to evaluate the decay of the Green function $(H_{[a,b]}-E)^{-1}(m,n)$ in terms of~$|m-n|$. Lemma~\ref{lem:ldt-to-decay} explains how the large deviation estimate from \cref{prop:ldt} can be used to guarantee the conditions
of Lemma~\ref{lem:Poisson}. This leads to the following \textbf{localization principle}: {\em the eigenfunction $\psi_j^{[a,b]}$ defined by
\[
H_{[a,b]}(x,\omega)\psi_j^{[a,b]}(x,\omega)=E_j^{[a,b]} (x, \omega)\psi_j^{[a,b]}(x,\omega)
\]
decays exponentially on any subinterval $[c,d]\subset [a,b]$ for which the large deviation estimate
\begin{equation}
\label{eq:2.locprinciple1}
\log \big | f_{[c,d]}(x,\omega,E_j^{[a,b]} (x, \omega)) \big | > (c-d)L(\omega,E_j^{[a,b]} (x, \omega)) - (c-d)^{1-\delta}
\end{equation}
is valid. }
This is of crucial importance to the theory of localization, and we shall make this precise later.

The following elementary observation links the spectra in finite volume to the decay of the Green function.

\begin{lemma}\label{lem:Poisson}
	Let $ x,\omega\in\T $, $ E\in\R $, and $ [a,b]\subset \Z $. If for any $ m\in[a,b] $, there exists
	$ \Lambda_m=[a_m,b_m]\subset[a,b] $ containing $m$ such that
	\begin{equation*}
		(1- \langle \delta_a,\delta_{a_m}\rangle) \left| (H_{\Lambda_m}(x,\omega)-E)^{-1}(a_m,m) \right|
		+(1- \langle \delta_b,\delta_{b_m}\rangle) \left| (H_{\Lambda_m}(x,\omega)-E)^{-1}(b_m,m) \right|
		<1,
	\end{equation*}
	then $ E\notin \spec H_{[a,b]}(x,\omega) $.
\end{lemma}

\begin{proof}
	Assume to the contrary that $ E\in \spec H_{[a,b]}(x,\omega) $ and let $ \psi $ be a corresponding
	eigenvector.  Let $ m\in [a,b] $ be such that $ |\psi(m)|=\max_n |\psi(n)|  $. The hypothesis
	together with the Poisson formula \cref{eq:Poisson} gives us that
	$ |\psi(m)|<\max(|\psi(a_m)|,|\psi(b_m)|) $ if $ a_m\neq a $ and $ b_m\neq b $,
	$ |\psi(m)|<|\psi(b_m)| $ if $ a_m=a $, and $ |\psi(m)|<|\psi(a_m)| $ if $ b_m=b $. In either case
	we reach a contradiction, so we must have $ E\notin \spec H_{[a,b]}(x,\omega) $.
\end{proof}

We use \cref{lem:Poisson} to establish our criterion for an energy to be in the spectrum. For this we will
also use the following well-known fact.

\begin{lemma}\label{lem:akbk}
	If for some $ x,\omega\in \T $, $ E\in \R $ there exist $ \delta>0 $ and sequences $ a_k\to-\infty $,
	$ b_k\to \infty $ such that
	\begin{equation*}
		\dist(E,\spec H_{[a_k,b_k]}(x,\omega))\ge \delta,
	\end{equation*}
	then
	\begin{equation*}
		\dist(E,\spec H(x,\omega))\ge \delta.
	\end{equation*}
\end{lemma}

\begin{proof}
	The hypothesis implies that for any $ \phi\in \ell^2(\Z) $ with finite support, there exists $ k $ such
	that
	\begin{equation*}
		\norm{(H(x,\omega)-E)\phi}=\norm{(H_{[a_k,b_k]}(x,\omega)-E)\phi}\ge \delta \norm{\phi}.
	\end{equation*}
	It follows by density that
	\begin{equation*}
		\norm{(H(x,\omega)-E)\phi}\ge\delta \norm{\phi}
	\end{equation*}
	for any $ \phi\in \ell^2(\Z) $, and this yields the conclusion.
\end{proof}

We can now formulate the \textbf{spectrum criterion} lemma. In the following two results, the notation $N^{1-}$ means $N^{1-\epsilon}$ for some small absolute $\epsilon>0$. For example $\epsilon=\frac{1}{100}$ will suffice (as in fact will large choices).

\begin{lemma}\label{lem:spectrum-criterion}
	Let $ \omega\in\T_{c,a} $, $ E\in \R $ be such that $ L(\omega,E)\ge \gamma>0 $. There exists
    $ N_0=N_0(V,c,a,\gamma) $ such that the following statement holds for any $ N\ge N_0 $.
    If for any $ x\in\T $, there exists $ r(x) \in [-N/2,N/2]$ such that
    \begin{equation*}
    \dist(E,\spec H_{r(x)+[-N,N]}(x,\omega))\ge \exp(-N^{1-}),
    \end{equation*}
    then
    \begin{equation*}
    	\dist(E,\cS_\omega)\ge \frac{1}{2}\exp(-N^{1-}).
    \end{equation*}
\end{lemma}

\begin{proof}
	Fix $ x\in \T $ and let $ \bar N\ge N $ be arbitrary. Let
	\begin{equation*}
		p=-\bar N-N+r(x-\bar N\omega), \ q=\bar N+N+r(x+\bar N\omega).
	\end{equation*}
	We will use \cref{lem:Poisson} to show that $ \tilde E\notin \spec H_{[p,q]}(x,\omega) $ for any $ |\tilde E-E|\le \exp(-N^{1-})/2 $. Note that, by \cref{prop:Lyapunov-continuity}, we have
	$ L(\tilde E,\omega)\ge\gamma/2 $. From the hypothesis we infer that
	\begin{equation*}
		\dist(\tilde E,\spec H_{[p,p+2N]})\ge \frac{1}{2}\exp(-N^{1-}).
	\end{equation*}
	It follows from \cref{lem:spec-to-decay}  that
	\begin{equation*}
		\left| (H_{[p,p+2N]}(x,\omega)-\tilde E)^{-1}(p+2N,m) \right|<1
	\end{equation*}
	for any $ m\in[p,p+N+[N/2]] $. Analogously one has
	\begin{equation*}
		\left| (H_{[q-2N,q]}(x,\omega)-\tilde E)^{-1}(q-2N,m) \right|<1
	\end{equation*}	
	for any $ m\in[q-N-[N/2],q] $. For $ m\in[p+N+[N/2],q-N-[N/2]] $, let
	\begin{equation*}
		a_m=m-N+r(x+m\omega),\ b_m=m+N+r(x+m\omega).
	\end{equation*}
	We clearly have $ [a_m,b_m]\subset[p,q] $. Using the hypothesis and
	\cref{lem:spec-to-decay}, we get
	\begin{equation*}
		\left| (H_{[a_m,b_m]}(x,\omega)-\tilde E)^{-1}(a_m,m) \right|
		+\left| (H_{[a_m,b_m]}(x,\omega)-\tilde E)^{-1}(b_m,m) \right|<1.
	\end{equation*}
	We can now apply \cref{lem:Poisson} to get that $ \tilde E\notin \spec H_{[p,q]}(x,\omega) $. Since
	this is true for any $ |\tilde E-E|\le \exp(-N^{1-})/2 $, it follows that
	\begin{equation*}
		\dist(E,\spec H_{[p,q]}(x,\omega))\ge \frac{1}{2}\exp(-N^{1-}).
	\end{equation*}
	Since $ \bar N $ was arbitrary, it follows that we can choose sequences $ a_k\to-\infty $ and
	$ b_k\to \infty $ such that
	\begin{equation*}
		\dist(E,\spec H_{[a_k,b_k]}(x,\omega))\ge \frac{1}{2}\exp(-N^{1-}).
	\end{equation*}
	The conclusion follows from \cref{lem:akbk}.
\end{proof}

The previous lemma relates the full spectrum $ \cS_\omega $ to the finite scale spectrum
\begin{equation*}
	\cS_{N,\omega}:=\bigcup_{x\in\T} \spec H_{[-N,N]}(x,\omega).
\end{equation*}
The proof of the lemma cannot be adjusted to give a relation between the finite scale spectra for different
scales. Instead, we will use the following weaker result.

\begin{lemma}[{\cite[Lem. 13.2]{GolSch11}}]\label{lem:finite-spectrum-criterion}
	Let $ \omega\in\T_{c,a} $, $ E\in \R $ be such that $ L(\omega,E)\ge \gamma>0 $. There exists
    $ N_0=N_0(V,c,a,\gamma) $ such that the following statement holds for any $ N\ge N_0 $.
    If
    \begin{equation*}
    \dist(E,\cS_{N,\omega})\ge \exp(-N^{1-}),
    \end{equation*}
    then
    \begin{equation*}
    	\dist(E,\cS_{\bar N,\omega})\ge \frac{1}{2}\exp(-N^{1-}),
    \end{equation*}
    for any $ \bar N\ge N $.
\end{lemma}

\begin{proof}
	The proof is analogous to that of \cref{lem:spectrum-criterion}. The only difference is that we now know that $ r(x)=0 $.
\end{proof}

\section{Stability of the Spectrum}

In this section we address the issue of how much of the finite scale spectrum $ \cS_{N,\omega} $ survives when we pass to a larger scale
$ \bar N $ or to the full scale.

\begin{lemma}\label{lem:SN-SNbar}
	Let $ \omega\in\T_{c,a} $ and assume $ L(\omega,E)\ge \gamma>0 $ for any $ E\in(E',E'') $. There exist
	$ c_0=c_0(V,c,a,\gamma) $ and
	$ N_0=N_0(V,c,a,\gamma) $ such that
	\begin{equation*}
		\mes \left(\cS_{N,\omega}\cap(E',E'')\setminus \cS_{\bar N,\omega}\right)\le \exp(-c_0N)
	\end{equation*}
	for any $ \bar N\ge N\ge N_0 $.
\end{lemma}

\begin{proof}
	Let $ N^{(k)}=N^{2^k} $ and assume $ \bar N\le N^{(1)} $.
	Let $ E=E_j^{[-N,N]}(x,\omega)\in \cS_{N,\omega}\cap(E',E'') $. Let $ \ell=[c_1N] $, with $ c_1\ll 1 $.
	By \cref{prop:ldt} we can find an
	interval $ n_0+[-\ell,\ell] $, $ |n_0|\le C\ell\ll N $, on which the large deviation estimate holds.
	By \cref{lem:ldt-to-decay} and the Poisson formula it follows that
	\begin{equation*}
		\left| \psi^{[-N,N]}_j(x,\omega;n) \right|\le \exp(-cN),\quad |n-n_0|\le \ell/2.
	\end{equation*}
	Let
	\begin{equation*}
		\xi_l(n) =\begin{cases}
			\psi_j^{[-N,N]}(x,\omega;n) &, n\in[-N,n_0]\\
			0 &, \text{ otherwise}
		\end{cases},\quad\quad
		\xi_r(n) =\begin{cases}
			\psi_j^{[-N,N]}(x,\omega;n) &, n\in[n_0,N]\\
			0 &, \text{ otherwise}
		\end{cases}.
	\end{equation*}
	Then $ \norm{\xi_l}\ge 1/2 $ or $ \norm{\xi_r}\ge 1/2 $. If $ \norm{\xi_l}\ge 1/2 $, then the fact that
	\begin{equation*}
		\norm{(H_{[-N,-N+2\bar N]}(x,\omega)-E)\xi_l}\le \exp(-cN)
	\end{equation*}
	implies $ \dist(E,\cS_{\bar N,\omega})\le 2\exp(-cN) $. The same conclusion holds if $ \norm{\xi_r}\ge 1/2 $.
	Since the finite spectra are unions of intervals, it follows that
	\begin{equation*}
		\mes \left(\cS_{N,\omega}\cap(E',E'')\setminus \cS_{\bar N,\omega}\right)
		\lesssim \bar N \exp(-cN)\le \exp(-cN/2).
	\end{equation*}
	Recall that so far we are assuming $ \bar N\le N^{(1)}=N^2 $.
	In general, we can find $ k $ such that $ N^{(k)}\le \bar N\le N^{(k+1)} $ and we have
	\begin{multline*}
		\mes \left(\cS_{N,\omega}\cap(E',E'')\setminus \cS_{\bar N,\omega}\right)\\
		\le \mes \left(\cS_{N,\omega}\cap(E',E'')\setminus \cS_{N^{(1)},\omega}\right)
		+ \mes\left(\cS_{N^{(1)},\omega}\cap(E',E'')\setminus \cS_{N^{(2)},\omega}\right)\\
		+\ldots
		+ \mes\left(\cS_{N^{(k)},\omega}\cap(E',E'')\setminus \cS_{\bar N,\omega}\right)\\
		\le \exp(-cN)+\exp(-cN^{(1)})+\ldots+\exp(-cN^{(k)})
		\le \exp(-cN/2).
	\end{multline*}
\end{proof}

If the mass of an eigenvector $ \psi_j^{[-N,N]} $ is concentrated near the edges of the interval, then we cannot guarantee
that the corresponding eigenvalue is close to $ \cS_\omega $. We can only come close to the full scale spectrum provided that
the mass of $ \psi_j^{[-N,N]} $ is concentrated inside the interval. It is not clear whether each $ E\in \cS_{N,\omega} $
can be associated with such an eigenvector. However, we can produce spectral segments of considerable size for which this holds.

\begin{lemma}\label{lem:initial-spectral-segment}
	Let $ \omega\in \T_{c,a} $, $ E\in \cS_\omega $ and assume $ L(E,\omega)\ge \gamma>0 $. There exist
	$ c_0(V,c,a,\gamma) $, $ C_0(V,c,a,\gamma) $, $ N_0(V,c,a,\gamma) $ such that the following
	statement holds for $ N\ge N_0 $. There exist $ x_0\in \T $ and $ j_0\in[-N,N] $ such that
	\begin{equation*}
		\left| E_{j_0}^{[-N,N]}(x_0,\omega)-E \right| \le \exp(-N^{1-}),
	\end{equation*}
	and for all
	$ |x-x_0|<\exp(-(\log N)^{C_0}) $, we have
	\begin{equation*}
		\left| \psi_{j_0}^{[-N,N]}(x,\omega;n) \right|
		\le \exp(-c_0 N), \ |n|\ge (1-c_0)N.
	\end{equation*}
\end{lemma}

\begin{proof}
	Since $ E\in \cS_\omega $, \cref{lem:spectrum-criterion} implies that there exists $ x'\in\T $ such
	that
	\begin{equation}\label{eq:E-spec-N}
		\max_{|n|\le N/2} \dist(E,\spec H_{n+[-N,N]}(x',\omega))\le \exp(-N^{1-}).
	\end{equation}
	Let $ \ell=[cN] $, with $ c<1 $ to be chosen later. We will argue that there exists $ |n_0|\le N/2 $ such
	that the Green function at scale $ \ell $ has off-diagonal decay on the intervals
	\begin{equation*}
		n_0+[-N,-N+\ell-1] \text{\ \ and\ \ } n_0+[N-\ell+1,N]
	\end{equation*}
	at the edges of $ n_0+[-N,N] $.
	  Due to \cref{prop:Wegner} we know that there
	exists $ A(V,c,a,\gamma)\gg 1 $ such that
	\begin{equation*}
		\left\{ x\in \T: \dist(E,\spec H_{\ell}(x_0,\omega)) <\exp(-(\log \ell)^A)\right\}
		\subset \bigcup_{k=1}^{k_0}I_k,
	\end{equation*}
	where $ I_k $ are intervals such that $ |I_k|\le \exp(-(\log \ell)^{A/2}) $, and $ k_0\le C\ell $. We now
	set $ c=(4C)^{-1} $ so that we have $ \ell\le N/4 $. Due to the Diophantine condition, each $ I_k $ contains
	at most one point of the form $$ x'+(n-N)\omega \text{\ \ or\ \ } x'+(n+N-\ell+1)\omega, $$ with $ |n|\le N/2 $.
	Since $ k_0\le N/4 $, it follows
	that there exists $ |n_0|\le N/2 $ such that $ x_0+(n_0-N)\omega $ and $  x_0+(n_0+N-\ell+1)\omega  $
	are not in any of the $ I_k $ and therefore
	\begin{multline}\label{eq:dist-E-H-N'}
		\dist(E,\spec H_{\ell}(x_0+(n_0-N)\omega,\omega)),
		\dist(E,\spec H_{\ell}(x_0+(n_0+N-\ell+1)\omega,\omega))\\
		\ge\exp(-(\log \ell)^A)\gg \exp(-N^{1-}).
	\end{multline}
	Let $ x_0=x'+n_0\omega $. By \cref{eq:E-spec-N} there exists $ j_0 $ such that
	\begin{equation}
		\left| E-E_{j_0}^{[-N,N]}(x_0,\omega) \right|\le \exp(-N^{1-}).
	\end{equation}
	From this, \cref{eq:dist-E-H-N'}, and \cref{eq:Hx-vs-Hx0} it follows that
	\begin{multline*}
		\dist(E_{j_0}^{[-N,N]}(x,\omega),\spec H_{\ell}(x-N\omega,\omega)),\\
		\dist(E_{j_0}^{[-N,N]}(x,\omega),\spec H_{\ell}(x+(N-\ell+1)\omega,\omega))
		\ge \frac{1}{2}\exp(-(\log \ell)^A),
	\end{multline*}
	for any $ |x-x_0|\le c \exp(-(\log \ell)^A) $. \cref{lem:spec-to-decay} implies
	that
	\begin{align*}
		& \left| (H_{\ell}(x-N\omega,\omega)-E_{j_0}^{[-N,N]}(x,\omega))^{-1}(j,k) \right|
		\le \exp(-\gamma|j-k|+(\log N)^C)\\
		& \left| (H_{\ell}(x+(N-\ell+1)\omega,\omega)-E_{j_0}^{[-N,N]}(x,\omega))^{-1}(j,k) \right|
		\le \exp(-\gamma|j-k|+(\log N)^C).
	\end{align*}
	The desired estimates on the eigenvector $ \psi_{j_0}^{[-N,N]}(x,\omega) $ now follow by applying
	the Poisson formula on the intervals $ [-N,-N+\ell-1] $ and $ [N-\ell+1,N] $.
\end{proof}

Next we address the stability of the spectral segments produced via the previous lemma. As in the proof of \cref{lem:SN-SNbar} we need
to argue by induction on scales. The inductive step that will be stated in \cref{lem:stabilization-approx-ef} is essentially Lemma~12.22 of
\cite{Bou05}. For the convenience of the reader we will sketch its proof. The original proof is for the case when the potential is a trigonometric polynomial. We will include the simple approximation argument needed to deal with analytic potentials. For this we recall some facts regarding the approximation of the potential by trigonometric polynomials.

Let
\begin{equation*}
	V(x)=\sum_{n=-\infty}^\infty v_n e(nx),
\end{equation*}
be the Fourier series expansion for $ V $, where we use the notation $ e(x):=\exp(2\pi i x) $. It is known that
since $ V $ is real-analytic on $ \T $, it can be extended to a strip of width $ 2\rho_0 $ around the real
axis, for some $ \rho_0>0 $, and that this implies the existence of $ C $ such that
\begin{equation*}
	|v_n|\le C\exp(-\pi \rho_0|n|).
\end{equation*}
In fact, we can take $ C=\sup_{x\in \T} |V(x\pm i\rho_0/2)| $. Let
\begin{equation}\label{eq:V_K}
	\tilde V(x)=\sum_{n=-K}^K v_n e(nx).
\end{equation}
As a consequence of the bound on the Fourier coefficients, we have
\begin{equation*}
	\sup_{x\in \T} |V(x)-\tilde V(x)|\le C(\norm{V}_\infty,\rho_0) \exp(-\pi\rho_0K/3).
\end{equation*}
It follows that we always have
\begin{equation}\label{eq:trig-approximation}
	\left| E_j^{(N)}(x,\omega)-\tilde E_{j}^{(N)}(x,\omega) \right|
	\le \norm{H_N(x,\omega)-\tilde H_{N}(x,\omega)}\le C\exp(-cK).
\end{equation}

\begin{lemma}\label{lem:stabilization-approx-ef}
	Let $ \omega\in\T_{c,a} $ and assume $ L(\omega,E)\ge \gamma>0 $ for any $ E\in(E',E'') $. Let $ I\subset [0,1] $ be an interval and
	let $ j\in[-N,N] $. Assume that $ E_j^{[-N,N]}(I,\omega)\subset (E',E'') $ and that for each $ x\in I $, there exists $ \xi $,
	$ \norm{\xi}=1 $, with support in $ [-N+1,N-1] $, such that
	\begin{equation}\label{eq:approx-ef}
		\norm{(H(x,\omega)-E_j^{[-N,N]}(x,\omega))\xi}<e^{-c_0N},
	\end{equation}
	where $ c_0>0 $ is some constant. Let
	\begin{equation*}
		\log N_1 \gg \log N\gg \log\log N_1.
	\end{equation*}
	If $ c_0\le C(V,c,a,\gamma)\ll 1 $ and $ N\ge N_0(V,c,a,\gamma,c_0) $, then we can partition $ I $ into intervals $ I_m $,
	$ m\le N_1^C $, with $ C $ an absolute constant, and for each $ I_m $, there exists $ j_1\in[-N_1,N_1] $ such that
	\begin{equation*}
		\left| E_j^{[-N,N]}(x,\omega)-E_{j_1}^{[-N_1,N_1]}(x,\omega) \right|\le \exp(-c_0N/2), x\in I_m,
	\end{equation*}
	and for each $ x\in I_m $, there exists $ \xi $, $ \norm{\xi}=1 $, with support in $ [-N_1+1,N_1-1] $, satisfying
	\begin{equation*}
		\norm{(H(x,\omega)-E_{j_1}^{[-N_1,N_1]}(x,\omega))\xi}<e^{-c_0N_1}.
	\end{equation*}	
\end{lemma}

\begin{proof}
	Fix $ x\in I $ and let $ \xi $ be as in \cref{eq:approx-ef}. We have
	\begin{align*}
		H_{[-N,N]}(x,\omega) &= H_{[-N_1,N_1]}(x,\omega),\\
		H_{[-N,N]}(x,\omega)\xi &= E_j^{[-N,N]}(x,\omega)\xi+O(e^{-c_0N})\\
		& = E_j^{[-N,N]}(x)\sum \langle \xi,\psi_k^{[-N_1,N_1]}(x,\omega)\rangle \psi_k^{[-N_1,N_1]}(x,\omega)+ O(e^{-c_0N}),\\
		H_{[-N_1,N_1]}(x,\omega)\xi &=
			\sum E_k^{[-N_1,N_1]}(x,\omega)\langle \xi,\psi_k^{[-N_1,N_1]}(x,\omega)\rangle \psi_k^{[-N_1,N_1]}(x,\omega).
	\end{align*}
	Note that for the first two identities, we used the fact that $ \xi $ is supported in $ [-N+1,N-1] $.
	It follows that
	\begin{equation}\label{eq:Ej-Ek}
		\left(\sum|E_j^{[-N,N]}(x,\omega)-E_k^{[-N_1,N_1]}(x,\omega)|^2
			\langle \xi,\psi_k^{[-N_1,N_1]}(x,\omega)\rangle^2 \right)^{1/2}<e^{-c_0N}.
	\end{equation}
	Since $ \norm{\xi}=1 $, there exists $ j_1(x)\in[-N_1,N_1] $ such that
	\begin{equation}\label{eq:psi-mass}
		| \langle \xi, \psi_{j_1}^{[-N_1,N_1]}(x,\omega)\rangle |\ge 1/\sqrt{N_1}.
	\end{equation}
	The estimate \cref{eq:Ej-Ek} implies that
	\begin{equation}\label{eq:Ej-Ej1}
		|E_j^{[-N,N]}(x,\omega)-E_{j_1}^{[-N_1,N_1]}(x,\omega)|<\sqrt{N_1}e^{-c_0N}.
	\end{equation}
	
	As in the proof of \cref{lem:initial-spectral-segment} it can be seen that there exists $ N_1/2<k_1<N_1 $ such that
	\begin{equation}\label{eq:psi-decay}
		|\psi_{j_1}^{[-N_1,N_1]}(x,\omega;n)|\le \exp(-3c_0N_1),  k_1\ge |n|\ge k_1-3c_0N_1 .
	\end{equation}
	For this we used the assumption that $ c_0 $ is small enough.
	Let $ \eta $ be the normalized projection of $ \psi_{j_1}^{[-N_1,N_1]}(x,\omega) $ onto the subspace corresponding to the
	interval $ [-k_1,k_1] $. By \cref{eq:psi-mass} and \cref{eq:psi-decay}, we have
	\begin{equation}\label{eq:existence-of-approx-ef}
		\norm{(H(x,\omega)-E_{j_1}^{[-N_1,N_1]})\eta}\lesssim \sqrt{N_1}\exp(-3c_0N_1)<\exp(-2c_0N_1).
	\end{equation}
	
	Now we just need to estimate the number of components of the set of phases $ x $ that satisfy \cref{eq:Ej-Ej1} and
	\cref{eq:existence-of-approx-ef}. For this we need to approximate the potential $ V $ by a trigonometric polynomial.
	For the purpose of the approximation, we note that the existence of $ \eta $, $ \norm{\eta}=1 $,
	with support in $ [-N_1+1,N_1-1] $, satisfying \cref{eq:existence-of-approx-ef} is equivalent to
	\begin{equation}\label{eq:existence-of-approx-ef-equiv}
		\norm{\left[P_{N_1}\left(H(x,\omega)-E_{j_1}^{[-N_1,N_1]}\right)^*
			\left(H(x,\omega)-E_{j_1}^{[-N_1,N_1]}\right)P_{N_1}\right]^{-1}}
		>\exp(2c_0N_1),
	\end{equation}
	where $ P_{N_1} $ is the projection onto the subspace corresponding to the interval $ [-N_1+1,N_1-1] $.
	Choose $ \tilde V $ as in \cref{eq:V_K} with $ K=CN_1^2 $. Then we have
	\begin{equation*}
		|\tilde E_j^{[-N,N]}(x,\omega)-\tilde E_{j_1}^{[-N_1,N_1]}(x,\omega)|<2\sqrt{N_1}e^{-c_0N},
	\end{equation*}
	and
	\begin{equation*}
		\norm{\left[P_{N_1}\left(\tilde H(x,\omega)-\tilde E_{j_1}^{[-N_1,N_1]}\right)^*
			\left(\tilde H(x,\omega)-\tilde E_{j_1}^{[-N_1,N_1]}\right)P_{N_1}\right]^{-1}}
		>\exp(2c_0N_1)/2.
	\end{equation*}	
	The set of $ x $'s satisfying the above estimates can be given a semialgebraic description in terms of polynomials of degree
	at most $ N_1^C $ (see the proof of \cite[Lem. 12.22]{Bou05}). It follows that $ I $ can be partitioned into intervals
	$ I_m $, $ m\le N_1^C $, such that $ j_1(x) $ from the above estimates can be kept constant on each of the subintervals. Going
	back to the original potential, the estimates \cref{eq:Ej-Ej1} and \cref{eq:existence-of-approx-ef} hold up to a correction
	by a constant factor, and with the constant choice of $ j_1 $ on each $ I_m $. This concludes the proof.
\end{proof}

The next lemma is our result on the stability of the spectral segments from \cref{lem:initial-spectral-segment}.

\begin{lemma}\label{lem:finite-segment-to-full-spectrum-Bourgain}
	Let $ \omega\in\T_{c,a} $ and assume $ L(\omega,E)\ge \gamma>0 $ for any $ E\in(E',E'') $. Let $ I\subset [0,1] $ be an interval and
	let $ j\in[-N,N] $. Assume that $ E_j^{[-N,N]}(I,\omega)\subset (E',E'') $ and that for each $ x\in I $, there exists $ \xi $,
	$ \norm{\xi}=1 $, with support in $ [-N+1,N-1] $, such that
	\begin{equation}\label{eq:approx-ef-bis}
		\norm{(H(x,\omega)-E_j^{[-N,N]}(x,\omega))\xi}<e^{-c_0N},
	\end{equation}
	where $ c_0>0 $ is some constant.
	If $ c_0\le C(V,c,a,\gamma)\ll 1 $ and $ N\ge N_0(V,c,a,\gamma,c_0) $, then
	\begin{equation*}
		\mes(E_j^{[-N,N]}(I,\omega)\setminus \cS_\omega)<\exp(-c_0N/4).
	\end{equation*}
\end{lemma}

\begin{proof}
	Let  $ N_1=N^2 $. Using \cref{lem:stabilization-approx-ef}, we
	partition $ I $ into intervals $ I_m $,
	$ m\le N_1^C $, and for each $ I_m $, there exists $ j_1\in[-N_1,N_1] $ such that
	\begin{equation}\label{eq:Ej-Ej1-bis}
		\left| E_j^{[-N,N]}(x,\omega)-E_{j_1}^{[-N_1,N_1]}(x,\omega) \right|\le \exp(-c_0N/2), x\in I_m,
	\end{equation}
	and for each $ x\in I_m $, there exists $ \xi $, $ \norm{\xi}=1 $, with support in $ [-N_1+1,N_1-1] $, satisfying
	\begin{equation}\label{eq:approx-ef-N1}
		\norm{(H(x,\omega)-E_{j_1}^{[-N_1,N_1]}(x,\omega))\xi}<e^{-c_0N_1}.
	\end{equation}	
 	Let
	\begin{equation*}
		\cE_{N,1,\omega}= \bigcup_m E_j^{[-N,N]}(I_m,\omega)\ominus E_{j_1}^{[-N_1,N_1]}(I_m,\omega),
	\end{equation*}
	where $ \ominus $ denotes the symmetric set difference. By the continuity of the parametrization of the eigenvalues and
	\cref{eq:Ej-Ej1-bis}, it follows that $ \mes(\cE_{N,1,\omega})\le\exp(-c_0N/3) $.
	
	Note that \cref{eq:approx-ef-bis} implies that $ \dist(E,\cS_\omega)<\exp(-c_0N) $ for all $ E\in E_j^{[-N,N]}(I,\omega) $.
	At the same time, if $ E\in E_j^{[-N,N]}(I,\omega)\setminus \cE_{N,1,\omega} $,
	then $ E\in E_{j_1}^{[-N_1,N_1]}(I_m,\omega) $, for some $ m $, and \cref{eq:approx-ef-N1} implies that
	$ \dist(E,\cS_\omega)<\exp(-c_0N_1) $.
	
	Let $ N_k=N^{2^k} $. Through iteration we obtain sets $ \cE_{N,k,\omega} $ such that
	$ \mes(\cE_{N,k,\omega})<\exp(-c_0 N_{k-1}/3) $, and if
	$ E\in E_j^{[-N,N]}(I,\omega)\setminus \bigcup_{l\le k}\cE_{N,l,\omega} $, then
	$ \dist(E,\cS_\omega)<\exp(-c_0N_k) $. Finally, we note that
	\begin{equation*}
		E_{j_0}^{[-N,N]}(I,\omega)\setminus \cS_\omega\subset \bigcup_k \cE_{N,k,\omega},
	\end{equation*}
	and we are done.
\end{proof}

\section{Proofs of \cref{thm:gaps-separation,thm:homogeneity}}


\begin{proof}[Proof of \cref{thm:gaps-separation}]
	First we prove the full scale statement.
	Let $ J $ be the interval between $ G_1 $ and $ G_2 $. Then there exists $ E_0\in \cS_\omega\cap J $. We will
	argue that the size of $ J $ is bounded below because
	$ \cS_\omega\cap J $ must be relatively large.
	\cref{lem:initial-spectral-segment} implies the existence of $ j\in[-N,N] $ and of a segment $ I $,
	$ |I|>\exp(-(\log N)^C) $, centered at a point $ x_0 $, such that
	\begin{equation}\label{eq:EN-E0}
		|E_{j}^{[-N,N]}(x_0,\omega)-E_0|\le\exp(-N^{1-}),
	\end{equation}
	and for any $ x\in I $ there exists $ \xi $, $ \norm{\xi}=1 $, with support in $ [-N+1,N-1] $, such that
	\begin{equation*}
		\norm{(H(x,\omega)-E_{j}^{[-N,N]}(x,\omega))\xi}<\exp(-cN)
	\end{equation*}
	(the vector $ \xi $ can be chosen to be the normalized projection of $ \psi_{j}^{[-N,N]}(x,\omega) $ onto the subspace
	corresponding to $ [-N+1,N-1] $). Using \cref{prop:Lyapunov-continuity} we can guarantee that $ L(\omega,E)\ge \gamma/2 $ for all
	$ E\in E_{j}^{[-N,N]}(I,\omega) $. Thus, we can apply \cref{lem:finite-segment-to-full-spectrum-Bourgain} to get
	\begin{equation*}
		\mes(E_j^{[-N,N]}(I,\omega)\setminus\cS_\omega)\le \exp(-cN).
	\end{equation*}
	By the continuity of $ E_j^{[-N,N]}(\cdot,\omega) $ and \cref{eq:EN-E0}, it follows that
	\begin{equation*}
		E_j^{[-N,N]}(I,\omega)\cap \cS_\omega\subset J\cap \cS_\omega
	\end{equation*}
	(otherwise, we would have $ \mes(E_j^{[-N,N]}(I,\omega)\setminus\cS_\omega)\ge \exp(-N^{1-})/2$ ).
	At the same time, \cref{cor:size-of-spectral-segment} implies
	\begin{equation*}
		\mes(E_j^{[-N,N]}(I,\omega))\ge \exp(-(\log N)^C).
	\end{equation*}
	Putting all these together we have
	\begin{equation*}
		|J|\ge |J\cap \cS_\omega| \ge |E_j^{[-N,N]}(I,\omega)\cap \cS_\omega| \ge \exp(-(\log N)^C)-\exp(-cN).
	\end{equation*}
	The conclusion follows immediately with an appropriate choice of $ C_0 $.
	
	The proof of the finite scale statement is analogous. One just needs to use \cref{lem:finite-spectrum-criterion} and
	\cref{lem:SN-SNbar} instead of \cref{lem:spectrum-criterion} and \cref{lem:finite-segment-to-full-spectrum-Bourgain}.
	As before, there exists $ E_0\in \cS_{\bar N,\omega}\cap J $.
	By \cref{lem:finite-spectrum-criterion}, there exist $ j\in[-N,N] $ and $ x_0\in \T $ such that
	\begin{equation}\label{eq:EN-E0-finite}
		|E_{j}^{[-N,N]}(x_0,\omega)-E_0|\le\exp(-N^{1-}).
	\end{equation}
	Let $ I=(x_0-\exp(-(\log N)^C),x_0+\exp(-(\log N)^C)) $. Using \cref{prop:Lyapunov-continuity} we can guarantee that $ L(\omega,E)\ge \gamma/2 $ for all
	$ E\in E_{j}^{[-N,N]}(I,\omega) $. Thus, we can apply \cref{lem:SN-SNbar} to get
	\begin{equation*}
		\mes(E_j^{[-N,N]}(I,\omega)\setminus\cS_{\bar N,\omega})\le \exp(-cN).
	\end{equation*}
	By the continuity of $ E_j^{[-N,N]}(\cdot,\omega) $ and \cref{eq:EN-E0-finite}, it follows that
	\begin{equation*}
		E_j^{[-N,N]}(I,\omega)\cap \cS_{\bar N,\omega}\subset J\cap \cS_{\bar N,\omega}
	\end{equation*}
	(otherwise, we would have $ \mes(E_j^{[-N,N]}(I,\omega)\setminus\cS_{\bar N,\omega})\ge \exp(-N^{1-})/2$ ).
	At the same time, \cref{cor:size-of-spectral-segment} implies
	\begin{equation*}
		\mes(E_j^{[-N,N]}(I,\omega))\ge \exp(-(\log N)^C).
	\end{equation*}
	Putting all these together we have
	\begin{equation*}
		|J|\ge |J\cap \cS_{\bar N,\omega}| \ge |E_j^{[-N,N]}(I,\omega)\cap \cS_{\bar N,\omega}| \ge \exp(-(\log N)^C)-\exp(-cN).
	\end{equation*}
	The conclusion follows immediately with an appropriate choice of $ C_0 $.
\end{proof}

\begin{proof}[Proof of \cref{thm:homogeneity}]
	We only prove the full scale version of the result. The finite scale statement is proved analogously as in the proof
	of \cref{thm:gaps-separation}. We only need to prove part (b). The other statements follow from its proof.
	Assume that $ \cS_\omega\cap(E',E'')\neq \emptyset $. Let $ E_0\in \cS_\omega \cap(E',E'') $. Let $ N\ge N_0 $ be
	large enough. 	From the proof of \cref{thm:gaps-separation} we know that there exist $ j\in[-N,N] $ and an interval $ I $,
	$ |I|>\exp(-(\log N)^C) $, centered at a point $ x_0 $, such that
	\begin{equation*}
		|E_{j}^{[-N,N]}(x_0,\omega)-E_0|\le\exp(-N^{1-}),
	\end{equation*}
	\begin{equation*}
		\mes(E_j^{[-N,N]}(I,\omega)\setminus\cS_\omega)\le \exp(-cN),
	\end{equation*}
	and
	\begin{equation*}
		\mes(E_j^{[-N,N]}(I,\omega))\ge \exp(-(\log N)^C).
	\end{equation*}

	Putting all this information together and assuming $ \exp(-(\log N)^C)<\epsilon $ (with $ \epsilon $ from \cref{eq:maximality}),
	we have
	\begin{multline*}
		|(E_0-\exp(-(\log N)^C),E_0+\exp(-(\log N)^C))\cap(\cS_\omega\cap(E',E''))|\\
		\ge \exp(-(\log N)^C)-\exp(-N^{1-})-\exp(-cN)\ge \frac{1}{2}\exp(-(\log N)^C).
	\end{multline*}
	Since this is true for arbitrary large enough $ N $, it follows that \cref{eq:1homogeneous} holds for $ \sigma\le \sigma_0 $,
	with $ \tau=1/4 $. The conclusion follows because \cref{eq:1homogeneous} holds trivially with $ \tau\simeq \sigma_0 $ for
	$ \sigma>\sigma_0 $.
\end{proof}

\section{Double resonances}

Of key importance in the theory of localization is the notion of a {\em double resonance}. This refers to the situation where inside of a large window $[-N,N]$, there are two smaller ones, say $I:=[k_1,k_2]$, $J:=[k_3,k_4]$, which are not too close and such that $H_I(x,\omega)$ and $H_J(x,\omega)$ have two eigenvalues $E_1, E_2$, respectively, with the property that $|E_1-E_2|$ is very small. If these eigenvalues correspond to eigenfunctions $H_I(x,\omega)\psi=E_1\psi$, $H_J(x,\omega)\phi=E_1\phi$, respectively, which are well localized within these respective windows, then $H_{[-N,N]}(x,\omega)$ exhibits an eigenvalue close to $E_1$ with two eigenfunctions $\psi \pm \phi$ (with the understanding that we set $\psi=0$ outside of $I$ and $\phi=0$ outside of $J$). It is a delicate matter to turn this idea into a quantitative, rigorous machinery.
Localization happens precisely if such long chains of resonances cannot occur.

For the almost Mathieu operator, double resonances can be handled explicitly via Lagrange interpolation for the trigonometric polynomial given by the finite-volume determinant. This is the method of Jitomirskaya \cite{Jit99}, which we will use in this section. Since the method does not apply to general potentials $V$, the problem is treated via elimination of ``exceptional frequencies'' $\omega$ in~\cite{GolSch11}. Alternatively, one can use  semi-algebraic techniques for the elimination of ``bad'' $\omega$. This technique is quite robust and applies for example to higher-dimensional tori. One can find a very effective and beautiful development of this method in the monograph~\cite{Bou05}.

For the purposes of this paper (as well as for the analysis of the gaps in \cite{GolSch11}), it is necessary to achieve a level of resolution in the double resonance problem that is considerably higher than the one required for localization. The reason for this lies with the distances between the eigenvalues on a finite interval. We use this separation to control the process of formation of the spectrum $\mathcal{S}_\omega$ on the whole lattice from the spectra on finite intervals (see \cref{sec:stabilization}). More specifically, to obtain points in $\mathcal{S}_\omega$ via $\mathcal{S}_N$ we need to keep the essential supports of the eigenfunctions in question bounded in order to obtain a spectral value of $H_\omega$ (see Remark~\ref{rem:2criterion}).
To keep the essential support compact, we make sure that the eigenfunction at scale $N$ gives rise to an eigenfunction at scale $\bar N\gg N$ that is very close to the initial function and with an eigenvalue that is very close to the initial eigenvalue (see \cref{prop:stabilization,prop:stabilization-AMO}).

To derive the quantitative separation for two Dirichlet eigenvalues, say $E_j^{(N)} (x,\omega)$ and $E_k^{(N)} (x,\omega)$, we need to verify that the respective sizes of the essential support of $\psi_j^{(N)} (x, \omega)$ and $\psi_k^{(N)} (x, \omega)$ are bounded by
\begin{equation}
\label{eq:2.supportsize1}
|\text{essential support}|\le \underline{N}:=Q_{N}:=\exp((\log\log  N)^C),
\end{equation}
$C\gg 1$. This is the estimate that allows one to evaluate
\begin{equation}\label{eq:2.apexpAg}
\bigl|f_N\bigl(x,\omega, E_1\bigr) -f_N\bigl(x,\omega, E_2\bigr) \bigr|
\end{equation}
from \cref{prop:ldt} for two close but distinct values $E_1,E_2$.
Heuristically, \cref{prop:uniform-upper-bound} states that the exceptional set in the large deviation estimate is close to an algebraic curve of
degree $\le (\log N)^A$. This level of resolution is fine enough to see the scale~\eqref{eq:2.supportsize1} in the setting of general potentials $V$ for which we use frequency modulation to eliminate the double resonances.

The next result, which follows from \cite[Prop. 5.5]{GolSch11} and \cref{prop:Wegner}, is a tool designed to obtain the desired resolution in the elimination of double resonances. This result employs the notion of {\em measure and complexity}. To be specific,
\begin{equation*}
	\mes(\cS)\le \epsilon,\quad \compl(\cS)\le K
\end{equation*}
means that for some intervals $ I_k $,
\begin{equation*}
	\cS\subset \bigcup_{k=1}^K I_k,\quad \sum_{k=1}^K |I_k|\le \epsilon.
\end{equation*}
Therefore, for the purposes of this paper we can assume that the sets from the next result are just unions of intervals.

\begin{propAlpha}\label{prop:elimination}
    Consider operators \eqref{eq:1.Sch} with real analytic $V$. Assume that $L( \omega, E)\ge\gamma  > 0$
    for any $\omega$ and any  $E\in(E', E'')$.
     There exists $\ell_0 = \ell_0(V,c,a,\gamma)$ such that for
    any $\ell_1 \ge \ell_2 \ge \ell_0$, the following holds: Given $t >
    \exp\bigl(\bigl(\log \ell_1\bigr)^{C_0}\bigr)$, $H \ge 1$, there
    exists a set $\Omega_{\ell_1, \ell_2, t, H} \subset \tor$ with
    \begin{equation*}
      \begin{aligned}
    \mes\bigl(\Omega_{\ell_1, \ell_2, t, H}\bigr) & <  \exp\big((\log\ell_1)^{C_1}\big) e^{-\sqrt H}\\
    \compl\bigl(\Omega_{\ell_1, \ell_2, t, H}\bigr) & < t
    \exp\big((\log\ell_1)^{C_1}\big) H
    \end{aligned}
    \end{equation*}
    such that for any $\omega \in \tor_{c,a} \setminus \Omega_{\ell_1,
    \ell_2, t, H}$ there exists a set $\cB_{\ell_1, \ell_2, t, H,
    \omega}\subset \mathbb{T}$ with
    \begin{align*}
    \mes\bigl(\cB_{\ell_1, \ell_2, t, H, \omega}\bigr) &
    	<  t\exp\big((\log\ell_1)^{C_1}\big)\, e^{-\sqrt{H}}\\
    \compl\bigl(\cB_{\ell_1, \ell_2, t, H, \omega}\bigr) & <  t
    \exp\big( (\log\ell_1)^{C_1}\big) H
    \end{align*}
    such that for any
    \[
    x\in \mathbb{T}\setminus \cB_{\ell_1, \ell_2, t, H, \omega},
    \]
    one has
    \begin{align*}
    \dist\bigl( \spec\bigl(H_{\ell_1} (x, \omega)\bigr)\cap (E',E'')
    ,\;
    \spec\bigl(H_{\ell_2}(x + t\omega,
    \omega)\bigr)\bigr) &\ge e^{-H(\log \ell_1)^{3C_2}}.
    \end{align*}
\end{propAlpha}

Even though no upper bound on the translation $t$ is stated here, note that the estimates are only meaningful
if $t<\bar{t}(H)$, where the latter makes the right-hand side in the measure estimate on the order of~$1$.
In view of Lemmas~\ref{lem:spec-to-ldt} and Lemma~\ref{lem:ldt-to-decay} it is natural to
recast the problem of double or multiple resonances by means of the following question:
{\em  for how many subintervals $[c,d]$ can the large deviation estimate
\begin{equation}
\label{eq:2.locprinciple2}
\log \big | f_{[c,d]}(x,\omega,E) ) \big | > (c-d)L(\omega,E) - (c-d)^{1-\delta}
\end{equation}
fail with the \textbf{same $x$, $E$}?
}
The large deviation estimate in \cref{prop:ldt}  tells us that the set where it fails is ``almost a curve"
in the plane of two variables $x$, $\omega$. Therefore a Bezout type argument should tell us that
two intervals occur only for special values of $x$, which is borne out by the previous proposition.

We shall now establish the following version of \cref{prop:elimination}
for the almost Mathieu operator. A key improvement over the general version is that no further elimination of
frequencies is required. However, this comes at the cost of the count of intervals where the large deviation
estimate might fail.

\begin{propAlphaM}\label{prop:elimination-AMO}
    Consider the almost Mathieu operator \eqref{eq:3.amopera} with $|\lambda|>1$ and $\omega\in \T_{c,a}$.
    Let $ \sigma>0 $.
    There exist $N_0=N_0(c,a,|\lambda|,\sigma)$ such that for $ N\ge N_0 $,
    $ 4N\le|t|\le \exp(N^{\sigma-}) $, and $ x\in \T $
    satisfying
    \begin{equation*}
    	\min_{|n|\le |t|+4N}\norm{x-\frac{n\omega}{2}}\ge \exp(-N^\sigma),
    \end{equation*}
    we have
    \begin{multline*}
    	\max \left(\max_{|n|\le[N/2]} \log|f_{n+[-N,N]}(x,\omega,E)|,
		\max_{|n|\le[N/2]+1} \log|f_{n+t+[-N,N]}(x,\omega,E)|\right)\\
		\ge (2N+1)L(\omega,E)-N^{\sigma+}.
    \end{multline*}
\end{propAlphaM}

The key idea is to invoke Lagrange interpolation for the determinants, because one can estimate the
size of the Lagrange basis polynomials (see \cref{lem:Lagrange-basis-estimate}).
It is clear from \eqref{eq:Dirichlet-det} that $ f_{[-N,N]} $ is an even trigonometric polynomial of degree
$ 2N+1 $:
\begin{equation*}
	f_{[-N,N]}(x,\omega,E)=\sum_{k=0}^{2N+1} a_k(\lambda,\omega,E)\cos^k 2\pi x=:Q(\cos 2\pi x).
\end{equation*}
Given $ \theta_k\in \T $, $ k=1,\ldots,2N+1 $, the Lagrange interpolation formula reads
\begin{equation}\label{eq:Lagrange-interpolation}
	Q(\cos 2\pi y)
	=\sum_{k=1}^{2N+1}Q(\cos 2\pi \theta_k)
		\frac{\prod_{n\neq k}(\cos 2\pi y-\cos 2\pi \theta_n)}
			{\prod_{n\neq k}(\cos 2\pi \theta_k-\cos 2\pi \theta_n)}.
\end{equation}
For the purposes of \cref{prop:elimination-AMO} we will take $ \{\theta_k\} $ to be made of two pieces
of the orbit of the irrational shift. We will control the size of the Lagrange basis polynomials by
invoking the following elementary estimate (cf. \cite[Lem. 3.1]{GolSch01}, \cite[Lem. 11]{Jit99}).

\begin{lemma}\label{lem:Koksma}
	Let $ \omega \in \T_{c,a} $, $ x,y\in \T $, $ f(\theta)=\log|\cos 2\pi y-\cos 2\pi \theta| $,
	$ \theta\in \T $. There exists $ C_0(a,c) $ such that
	\begin{equation*}
		\left| \sum_{k=1}^n f(x+k\omega)-n\int_0^1 f(\theta)\,d\theta \right|
		\lesssim n\delta \log\frac{1}{\delta}+C_0(\log n)^{a+2}\log\frac{1}{\delta}
	\end{equation*}
	for any $ n>1 $ and any $ 0<\delta\ll 1$ such that $ \delta\le \min_{k}\norm{(x+k\omega\pm y)/2} $.
\end{lemma}

\begin{proof}
	Let
	\begin{equation*}
		g_\delta^{\pm}(\theta)= \begin{cases}
			\log\left|\sin 2\pi\norm{(y\pm\theta)/2}\right|&,\norm{(y\pm\theta)/2}\ge \delta\\
			\log|\sin 2\pi\delta| &,\norm{(y\pm\theta)/2}<\delta
		\end{cases}.
	\end{equation*}
	If $ \delta\le \min_k\norm{(x+k\omega\pm y)/2} $, then
	\begin{equation*}
		\sum_{k=0}^n f(x+k\omega)
		=n\log 2+
			\sum_{k=1}^n g_\delta^{+}(x+k\omega)+\sum_{k=1}^n g_\delta^{-}(x+k\omega).
	\end{equation*}
	By Koksma's inequality (see \cite[Thm. 2.5.1]{KuiNie74}) we have
	\begin{equation*}
		\left| \sum_{k=1}^n g_\delta^{\pm}(x+k\omega)-n\int_0^1 g_\delta^\pm(\theta)\,d\theta \right|
		\le n D_n \textrm{Var}(g_\delta^{\pm})
		\lesssim C(a,c) (\log n)^{a+2} \log\frac{1}{\delta}.
	\end{equation*}
	The discrepancy $ D_n $ is evaluated via the Erd\"os-Turan theorem (see
	\cite[Lem. 2.3.2-3]{KuiNie74}). Finally, one has
	\begin{equation*}
		\left| \int_0^1 f(\theta)\,d\theta-\log 2-\int_0^1 g_\delta^{+}(\theta)\,d\theta
			-\int_0^1 g_\delta^{-}(\theta)\,d\theta \right|
		\lesssim \delta\log\frac{1}{\delta}.
	\end{equation*}
	and the lemma follows.
\end{proof}

\begin{lemma}\label{lem:Lagrange-basis-estimate}
	Let $ \omega\in \T_{c,a} $, $ x,y\in \T $, $ N\ge 1 $, $ t\in \Z $, and
	\begin{equation*}
		\theta_n= \begin{cases}
			x+ \left(-N+n- \left[\frac{N}{2}\right]-1\right)\omega
				&, n=1,\ldots,2\left[\frac{N}{2}\right]+1\\
			x+ \left(-N+t+n- \left[\frac{N}{2}\right]-N-1\right)\omega
				&, n=2\left[\frac{N}{2}\right]+2,\ldots,2N+1\\			
		\end{cases}.
	\end{equation*}
	Given $ \sigma>0 $, there exists $ N_0(c,a,\sigma) $ such that if we have
	\begin{equation*}
		\min_{|n|\le |t|+4N} \norm{x-\frac{n\omega}{2}}\ge \exp(-N^\sigma),
			\quad \min_{|n|\le |t|+4N} \norm{\frac{x+n\omega\pm y}{2}}\ge \exp(-N^\sigma)
	\end{equation*}
	for $ N\ge N_0 $ and $ 4N \le|t|\le \exp(-N^{\sigma-}) $, then
	\begin{equation*}
		\left| \frac{\prod_{n\neq k}(\cos 2\pi y-\cos 2\pi \theta_n)}
			{\prod_{n \neq k}(\cos 2\pi \theta_k-\cos 2\pi \theta_n)} \right|
				\le \exp(N^{\sigma+})
	\end{equation*}
	for any $ k=1,\ldots,2N+1 $.
\end{lemma}

\begin{proof}
    Recall that
    \begin{equation*}
    	\int_0^1 \log|\xi-\cos 2\pi \theta|\,d\theta=-\log 2, |\xi|\le 1.
    \end{equation*}
	Our assumptions on $ x $, $ y $, and $ t $, together with the Diophantine condition on $ \omega $,
	guarantee that we can apply \cref{lem:Koksma} with
	$ \delta=\exp(-N^\sigma) $ to get
	\begin{align*}
		& \left| \log\left|\prod_{n\neq k}(\cos 2\pi y-\cos 2\pi \theta_n)\right|
			+2N\log 2 \right|\le N^{\sigma+},\\
		& \left| \log\left|\prod_{n\neq k}(\cos 2\pi \theta_k-\cos 2\pi \theta_n)\right|
			+2N\log 2 \right|\le N^{\sigma+}.
	\end{align*}
	Note that the Diophantine condition and the assumption that $ |t|\le \exp(N^{\sigma-}) $ are
	needed to ensure that $ \norm{(\theta_k-\theta_n)/2}\ge \exp(-N^{\sigma}) $.
	The conclusion follows immediately.
\end{proof}

\begin{proof}[Proof of Proposition \ref{prop:elimination-AMO}]
    Fix $ x\in\T $ satisfying the assumptions of the proposition. Due to the large deviation estimate of
    \cref{prop:ldt}, we know there exists $ y\in \T $ satisfying the assumptions of
    \cref{lem:Lagrange-basis-estimate} and such that
    \begin{equation*}
    	\log|f_{[-N,N]}(y,\omega,E)|>(2N+1)L(\omega,E)-(\log N)^C.
    \end{equation*}
    At the same time, from \cref{lem:Lagrange-basis-estimate} and \cref{eq:Lagrange-interpolation}, we have
    \begin{equation*}
    	|f_{[-N,N]}(y,\omega,E)|=|Q(\cos 2\pi y)|\le (2N+1) \exp(N^{\sigma+})\max_n|Q(\cos 2\pi\theta_n)|.
    \end{equation*}
    Therefore, it follows that
    \begin{equation*}
    	\max_n \log|f_{[-N,N]}(\theta_n,\omega,E)|\ge (2N+1)L(\omega,E)-(N^{\sigma+})^{1-}.
    \end{equation*}
    This yields the desired conclusion.
\end{proof}

\section{Localized Eigenfunctions on Finite Intervals}

We now continue by proving results on finite scale localization of eigenfunctions. We give a detailed proof for the
almost Mathieu case and only briefly discuss the proof for the case of a general analytic potential. For the latter case, a
slightly different statement with detailed proof can be found in \cite{GolSch08}.

The result which we present here is adjusted to our criterion for identifying finite scale energies that are close to the full spectrum, as stated in
\cref{lem:spectrum-criterion}. In turn, this criterion is adapted to the elimination of resonances
afforded by \cref{prop:elimination-AMO}. Due to the weaker elimination of resonances, in the case of the
almost Mathieu operator we cannot immediately exclude the possibility of the eigenvector having some mass concentrated at the edges of a
given interval. Instead we will see that we can work around this issue by shifting the edges of the
interval. The shift is phase and energy dependent, and it will be crucial for \cref{prop:stabilization-AMO}
that our result addresses the stability of the shift.

\addtocounter{propAlpha}{1}
\begin{propAlphaM}\label{prop:localization-AMO}
    Consider the almost Mathieu operator \eqref{eq:3.amopera} with $|\lambda|>1$ and
    $\omega\in \mathbb{T}_{c,a}$. Let
    \begin{equation*}
    	\cB_{\ell,M,\omega}:= \left\{ x\in \T: \min_{|n|\le M+C_0 \ell} \norm{x-\frac{n\omega}{2}}
			<\exp(-\ell^{1/2})  \right\}.
    \end{equation*}
    There exist $\ell_0(|\lambda|,c,a)$, $ c_0(|\lambda|,c,a) $, $ C_0(|\lambda|,c,a) $,
    such that the following statement holds for any $\ell\ge \ell_0$, $ 4\ell\le M\le \exp(\ell^{c_0}) $.
   	Given $ E_0\in \R $ and $ x_0\in\R $,
    there exists $ N=N(x_0,E_0,\ell) $ such that $ 0\le N-M\le C_0 \ell $ and if
    \EQ{
    &	x\in(x_0-\exp(-(\log \ell)^{C_0}),x_0+\exp(-(\log \ell)^{C_0}))\setminus \cB_{\ell,M,\omega}, \\
	&	\left| E_j^{[-N,N]}(x,\omega)-E_0 \right|  \le \exp(-(\log \ell)^{C_0}), \\
    	& \max_{|s|\le[\ell/2]} \dist(E_j^{[-N,N]}(x,\omega),\spec H_{s+[-\ell,\ell]}(x,\omega))
		 \lesssim \exp(-\ell^{1/2+}),
    }
    then
    \begin{equation*}
    	\left| \psi_j^{[-N,N]}(x_0,\omega;n) \right|\le \exp\left(-|n|\log|\lambda|
		+C_0 \ell\right), |n|>4 \ell.
    \end{equation*}
\end{propAlphaM}
\addtocounter{propAlpha}{-1}

\begin{proof}
	From \cref{prop:Wegner} we know that there
	exists $ A(V,c,a,\gamma) $ such that
	\begin{equation*}
		\left\{ x\in \T: \dist(E_0,\spec H_{[-\ell,\ell]}(x,\omega)) <\exp(-(\log \ell)^A)\right\}
		\subset \bigcup_{k=1}^{k_0}I_k,
	\end{equation*}
	where $ I_k $ are intervals such that $ |I_k|\le \exp(-(\log \ell)^{A/2}) $, and $ k_0\le C\ell $. Due to the
	Diophantine condition, each interval $ I_k $ contains at most one point of the form $$ x_0+(-M-n+\ell)\omega \text{\
	or \ } x+(M+n-\ell)\omega, $$ respectively, with $ |n|\le \ell^C $. It follows that there exists $ |n_0|\le C\ell $ such that
	\begin{equation}\label{eq:n0}
		x_0+(-M-n_0+\ell)\omega,x_0+(M+n_0-\ell)\omega \notin \bigcup_{k=1}^{k_0}I_k.
	\end{equation}
	We let $ N=M+n_0$.
	Suppose $$ x\in (x_0-\exp(-(\log \ell)^{2A}),x_0+\exp(-(\log \ell)^{2A}))$$ and $ E=E_j^{[-N,N]}(x,\omega) $ are
	such that $ |E-E_0|\le \exp(-(\log \ell)^{2A}) $. From \cref{eq:n0} and \cref{eq:Hx-vs-Hx0} it follows that
	\begin{equation*}
		\dist(E,\spec H_{[-N,-N+2\ell]}(x,\omega)), \dist(E,\spec H_{[N-2\ell,N]}(x,\omega))
		\gtrsim \exp(-(\log \ell)^A).
	\end{equation*}	
	We will use \cref{lem:Poisson} to show that any $$ E'\in(E-\exp(-\ell^{1/2+}),E+\exp(-\ell^{1/2+}))$$
	is not in the spectrum of $ H $ on certain large intervals. Note that we have
	\begin{equation*}
		\dist(E',\spec H_{[-N,-N+2\ell]}(x,\omega)), \dist(E',\spec H_{[N-2\ell,N]}(x,\omega))
		\gtrsim \exp(-(\log \ell)^A).
	\end{equation*}
	\cref{lem:spec-to-ldt} and \cref{lem:ldt-to-decay} imply that
	\begin{equation*}
		\left| \left(H_{[-N,-N+2\ell]}(x,\omega)-E'\right)^{-1}(m,-N+2\ell) \right|<1,\ m\in[-N,-N+2\ell-\ell/4],
	\end{equation*}
	\begin{equation*}
		\left| \left( H_{[N-2\ell,N]}(x,\omega)-E'\right)^{-1}(m,N-2\ell) \right|<1,\ m\in [N-2\ell+\ell/4,N].
	\end{equation*}
	Since we also have
	\begin{equation*}
		\max_{|s|\le[\ell/2]}\dist(E',\spec H_{s+[-\ell,\ell]}(x,\omega))\lesssim \exp(-\ell^{1/2+}),
	\end{equation*}
	\cref{lem:ldt-to-decay} implies that
	\begin{equation*}
		\max_{|s|\le[\ell/2]}\log|f_{s+[-\ell,\ell]}(x,\omega,E')|\le (2\ell+1)L(E',\omega)-\ell^{1/2+}.
	\end{equation*}
	\cref{prop:elimination-AMO} and \cref{lem:ldt-to-decay} imply that for any
	$$ m\in[-N+2\ell-\ell/4,-4\ell]\cup[4\ell,N-2\ell+\ell/4] $$ there exists $ \Lambda_m=[a_m,b_m]\subset [-N,N] $ containing $m$ such
	that $ m-a_m,b_m-m>\ell/4 $ and
	\begin{equation*}
		\left| \left(H_{\Lambda_m}(x,\omega)-E'\right)^{-1}(m',a_m) \right|\ll 1, m'-a_m\ge \ell/4,
	\end{equation*}
	\begin{equation*}
		\left| \left(H_{\Lambda_m}(x,\omega)-E'\right)^{-1}(m',b_m) \right|\ll 1, b_m-m'\ge \ell/4.
	\end{equation*}
	\cref{lem:Poisson} now implies that $ E' $ is not in the spectrum of $ H $ restricted to $ [-N,b_{-4\ell}] $
	and $ [a_{4\ell},N] $. Since this is true for any $ E'\in(E-\exp(-\ell^{1/2+}),E+\exp(-\ell^{1/2+})) $, it
	follows that
	\begin{equation*}
		\dist(E,\spec H_{[-N,b_{-4\ell}]}(x,\omega)),\dist(E,\spec H_{[a_{4\ell},N]}(x,\omega))
		\gtrsim \exp(-\ell^{1/2+}).
	\end{equation*}
	\cref{lem:spec-to-ldt} implies that
	\begin{multline}\label{eq:ldt-localization}
		\log|f_{[-N,b_{-4\ell}]}(x,\omega,E)|,\log|f_{[a_{4\ell},N]}(x,\omega,E)|\\
		\ge NL(\omega,E)-C\ell-\ell^{1/2+}(\log N)^C\ge NL(\omega,E)-C'\ell,
	\end{multline}
	provided that $ N\le \exp(\ell^{c_0}) $ with $ c_0 $ small enough.
	The conclusion follows by using \cref{lem:ldt-to-decay} and the Poisson
	formula.
\end{proof}

We will now state the analogous result for general potentials.

\begin{propAlpha}\label{prop:localization}
    Consider the Schr\"odinger operator \cref{eq:1.Sch} with real analytic $ V $. Assume that
    $ L(\omega,E)\ge \gamma>0 $ for any $ \omega $ and any $ E\in(E',E'') $.
    There exist $\ell_0(V,c,a,\gamma)$, $ C_0(V,c,a,\gamma) $, $ C_1(V,c,a,\gamma) $
    such that the following statement holds for any $\ell\ge \ell_0$,
    $ N\ge \exp((\log \ell)^{2C_0})$, $ \omega\in \T_{c,a}\setminus \Omega_{\ell,N} $, and
    $ x\in \T\setminus \cB_{\ell,N,\omega} $, where, using the notation of \cref{prop:elimination}, we have
    \begin{equation*}
    	\Omega_{\ell,N}=\bigcup_{\exp((\log \ell)^{C_0})\le|t|\le N} \Omega_{2\ell+1,2\ell+1,t,\ell^{1/2}},\quad
		\cB_{\ell,N,\omega}=\bigcup_{\exp((\log \ell)^{C_0})\le|t|\le N,|s|\le \ell} (s\omega+\cB_{2\ell+1,2\ell+1,t,\ell^{1/2},\omega}).
    \end{equation*}
   	If $ E_j^{[-N,N]}(x,\omega)\in (E',E'') $ is such that
    \begin{equation}\label{eq:2nonresMNEW1}
    	\max_{|s|\le[\ell/2]} \dist(E_j^{[-N,N]}(x,\omega),H_{s+[-\ell,\ell]}(x,\omega))
			\lesssim \exp(-\ell^{1/2+}),
    \end{equation}
    then
    \begin{equation*}
    	\left| \psi_j^{[-N,N]}(x,\omega;n) \right|\le \exp\left(-|n|\gamma
		+C_1\exp((\log\ell)^{C_0})\right), |n|>\exp((\log \ell)^{C_0}).
    \end{equation*}
\end{propAlpha}

\begin{proof}
	The main idea is to use \cref{lem:Poisson} to show that any
	\begin{equation*}
		E\in\left(E_j^{[-N,N]}(x,\omega)-\exp(-\ell^{1/2+}),
			E_j^{[-N,N]}(x,\omega)+\exp(-\ell^{1/2+})\right)
	\end{equation*}
	is not in the spectrum of $ H $ on certain large intervals. Since we have
	\begin{equation*}
		\max_{|s|\le[\ell/2]}\dist(E,\spec H_{s+[-\ell,\ell]}(x,\omega))\lesssim \exp(-\ell^{1/2+}),
	\end{equation*}
	\cref{prop:elimination} implies that
	\begin{equation*}
		\dist(E,\spec H_{t+[-\ell,\ell]}(x,\omega))\gtrsim \exp(-\ell^{1/2} (\log \ell)^{C_0}),\quad
			\exp((\log \ell)^C)<|t|\le N.
	\end{equation*}
	By the same reasoning as in the proof of \cref{prop:localization-AMO}, we obtain
	\begin{multline*}
		\dist \left(E_j^{[-N,N]}(x,\omega), \spec H_{[-N,-a]}(x,\omega)\right),\\
		\dist \left(E_j^{[-N,N]}(x,\omega), \spec H_{[a,N]}(x,\omega)\right)
		\gtrsim  \exp(-\ell^{1/2} (\log \ell)^{C_0}),
	\end{multline*}
	with $ a=[\exp((\log \ell)^{C_0})]+1 $.
	The conclusion follows by using \cref{lem:spec-to-ldt},  \cref{lem:ldt-to-decay}, and the Poisson
	formula.	
\end{proof}

Note that, due to \cref{prop:elimination}, in the above proposition we have
\begin{equation}\label{eq:complexity-estimates}
	\begin{aligned}
	& \mes(\Omega_{\ell,N})\le N\exp(-\ell^{1/4}/2), & &\compl(\Omega_{\ell,N})\le N^2\exp((\log \ell)^C),\\
	& \mes(\cB_{\ell,N,\omega})\le N^2\exp(-\ell^{1/4}/2), & &
	\compl(\cB_{\ell,N,\omega})\le N^2 \exp(\exp(\log \ell)^C),
	\end{aligned}
\end{equation}
provided $ \ell $ is large enough. This shows that the above result is meaningful as long as $ N\le \exp(l^\epsilon) $.

\begin{remark}\label{rem:2criterion}\upshape
    Condition \eqref{eq:2nonresMNEW1} in \cref{prop:localization,prop:localization-AMO} implies that
    the essential support of the eigenfunction remains close to the origin as
    $N$ grows.
    This condition serves as a criterion for a given value $E_0$ to fall into the
    spectrum in the regime of positive
    Lyapunov exponents.
    This is the meaning of~\cref{lem:spectrum-criterion}. Let us note that the elimination of
    resonances in \cref{prop:elimination,prop:elimination-AMO}
    combined with the Poisson formula ensures only that the essential support of the eigenfunction
    cannot be too spread out.
    However, this obviously does not specify where the essential support is located.
\end{remark}

\section{Separation of Finite Scale Eigenvalues}\label{sec:separation}

Next we discuss the separation of finite scale eigenvalues. The basic idea is that if two distinct eigenvalues are too close, then we can show that their corresponding eigenfunctions are also close, contradicting their orthogonality. It follows from \cref{eq:transfer-and-determinants} that the eigenvector $ \psi  $ for the Dirichlet problem on $ [a,b] $, normalized by $ \psi(a)=1 $, is given by
\begin{equation}\label{eq:eigenvector-determinant}
	\psi(n)=f_{[a,n-1]}(x,\omega,E),\qquad  n\in[a,b],
\end{equation}
with the convention that $ f_{[a,a-1]}=1 $. Thus, we can estimate the distance between the eigenvectors corresponding to different energies by using the following consequence of the uniform upper bound estimate.

\begin{corollary}[{\cite[Cor. 2.14]{GolSch11}}]\label{cor:uniform-Lipschitz}
 	Fix $ \omega_0\in \T_{c,a} $ and $ E_0\in \C $. Assume that
	$ L(\omega_0,E_0)\ge \gamma>0 $. Let $ \partial $ denote one of the partial derivatives
	$ \partial_x $, $ \partial_E $, $ \partial_\omega $. Then
	\begin{equation*}
		\sup \left\{ \log\norm{\partial M_N(x,\omega,E)}: |E-E_0|+|\omega-\omega_0|<N^{-C}, x\in\T \right\}
		\le NL(\omega_0,E_0)+C(\log N)^{C_0},
	\end{equation*}
	for all $ N\ge 2 $. Here $ C_0=C_0(a) $ and $ C=C(V,a,c,\gamma,E_0) $.
\end{corollary}

We are ready to prove separation of eigenvalues for the almost Mathieu operator.
\addtocounter{propAlpha}{1}

\begin{propAlphaM}\label{prop:separation-AMO}
	Using the notation and the assumptions of \cref{prop:localization-AMO}, we have that
	\begin{equation*}
		|E_k^{[-N,N]}(x,\omega)-E_j^{[-N,N]}(x,\omega)|>\exp(-C_2 \ell), \quad k\neq j
	\end{equation*}
	with $ C_2=C_2(|\lambda|,c,a)\gg C_1 $.
\end{propAlphaM}
\addtocounter{propAlpha}{-1}

\begin{proof}
	We argue by contradiction. Let
	\begin{equation*}
		E_1=E_j^{[-N,N]}(x,\omega),\quad E_2=E_k^{[-N,N]}(x,\omega),
	\end{equation*}
	and assume
	that $ |E_1-E_2|<\exp(-C_2 \ell) $. We have $ |E_2-E_0|\lesssim \exp(-\ell^{1/2}) $, so
	\cref{prop:localization-AMO} applies to $ E_2 $ also. We know from \cref{eq:eigenvector-determinant} that
	$$ \psi_i(n)=f_{[-N,n-1]}(x,\omega,E_i),\qquad i=1,2 $$ are eigenvectors corresponding to  $ E_1 $ and
	$ E_2 $. \cref{prop:localization-AMO} implies that
	\begin{equation*}
		\sum_{|n|>C\ell} |\psi_i(n)|^2\le \exp(-C\ell\log|\lambda|)\sum |\psi_i(n)|^2,
	\end{equation*}
	provided $ C\gg C_1 $. From \cref{cor:uniform-Lipschitz} it follows that
	\begin{equation*}
		\sum_{|n|\le C\ell} |\psi_1(n)-\psi_2(n)|^2\le |E_1-E_2|^2\exp(2NL(E_1,\omega)+C\ell)
		\le \exp(-2C_2 \ell+2NL(E_1,\omega)+C\ell).
	\end{equation*}
	From \cref{eq:ldt-localization} we know that
	\begin{equation*}
		\sum |\psi_1(n)|^2\ge \exp(2NL(E_1,\omega)-C\ell).
	\end{equation*}
	Therefore we have
	\begin{equation*}
		\sum_{|n|\le C\ell} |\psi_1(n)-\psi_2(n)|^2\le \exp(-C_2\ell)\sum |\psi_1(n)|^2,
	\end{equation*}
	provided $ C_2 $ is large enough.
	We arrive at the estimate
	\begin{equation*}
		\sum |\psi_1(n)|^2+\sum |\psi_2(n)|^2=\norm{\psi_1-\psi_2}^2
		\le \exp(-C\ell) \left(\sum |\psi_1(n)|^2+\sum |\psi_2(n)|^2\right).
	\end{equation*}
	This is impossible and concludes the proof.
\end{proof}

We only state the analogous result for general analytic potentials. Its proof is completely analogous to \cref{prop:separation-AMO}.
The difference in the results comes from the difference in the sizes of the localization windows. Note that, for this reason, the separation is much better for the almost Mathieu operator.

\begin{propAlpha}\label{prop:separation}
	Using the notation and the assumptions of \cref{prop:localization}, we have that
	\begin{equation*}
		|E_k^{[-N,N]}(x,\omega)-E_j^{[-N,N]}(x,\omega)|>\exp(-C_2 \exp((\log \ell)^{C_0})), \quad  k\neq j
	\end{equation*}
	with $ C_2=C_2(V,c,a,\gamma)\gg C_1 $.
\end{propAlpha}

\section{Stabilization of Finite Scale Spectral Segments}\label{sec:stabilization}

\cref{prop:separation} and \cref{prop:separation-AMO} allow us to obtain a {\em stability property} of the finite volume spectra as we pass from one scale to the next bigger one. This paves the way for a multi-scale control of the spectrum in infinite volume.

We first recall some well-known estimates on the stabilization of finite scale eigenvalues and eigenfunctions as the scale increases.

\begin{lemma}\label{lem:eigenvalue-stabilization}
	Let $ x,\omega\in \T $. For any intervals $ \Lambda_0=[a_0,b_0]\subset \Lambda\subset \Z $ and any
	$ j_0 $, we have
	\begin{equation*}
		\dist\left(E_{j_0}^{\Lambda_0}(x,\omega),\spec H_\Lambda(x,\omega)\right)
		\le \left| \psi_{j_0}^{\Lambda_0}(x,\omega;a_0) \right|
			+\left| \psi_{j_0}^{\Lambda_0}(x,\omega;b_0) \right|.
	\end{equation*}
\end{lemma}

\begin{proof}
	Let $ \psi_0 $ be the extension, with zero entries, of $ \psi_{j_0}^{\Lambda_0}(x,\omega) $ to
	$ \Lambda $. Since $ \norm{\psi_0}=1 $, the conclusion follows from the fact that we have
	\begin{equation*}
		\norm{(H_\Lambda(x,\omega)-E_{j_0}^{\Lambda_0}(x,\omega))\psi_0}
		\le \left| \psi_{j_0}^{\Lambda_0}(x,\omega;a_0) \right|
			+\left| \psi_{j_0}^{\Lambda_0}(x,\omega;b_0) \right|.
	\end{equation*}
	Indeed, this implies that $\|(H_\Lambda(x,\omega)-E_{j_0}^{\Lambda_0}(x,\omega))^{-1}\|^{-1}$ is also bounded
	by the right-hand side and the lemma follows by self-adjointness of $H_\Lambda(x,\omega)$.
\end{proof}

\begin{lemma}\label{lem:eigenvector-stabilization}
    Let $ A $ be a finite dimensional Hermitian operator.
    Let $E,\eta\in \mathbb{R}$, $\eta>0$. Assume that the subspace of the eigenvectors
    of $A$ with eigenvalues falling into the interval $(E-\eta,E+\eta)$ is at most of dimension one.
    If there exists $ \phi $ such that $ \norm{\phi}=1 $ and
    \begin{equation*}
	    \|(A-E)\phi\|< \epsilon<\eta,
    \end{equation*}
    then there exists an eigenvector $\psi_0$ with an eigenvalue $E_0\in (E-\varepsilon,E+\varepsilon)$, such that
    \begin{equation*}
        \|\phi-\psi_0\|\lesssim \epsilon\eta^{-1}
    \end{equation*}
\end{lemma}

\begin{proof}
    Let $\{\psi_j\}$ be an orthonormal basis of eigenvectors of $A$, $A\psi_j=E_j\psi_j$.
    Then
    \begin{equation*}
	    \epsilon^2>\|(A-E)\phi\|^2=\sum_{j}|\langle \phi,\psi_j\rangle|^2(E_j-E)^2\ge \min_j (E_j-E)^2\\
    \end{equation*}
    This implies that there exists $E_k\in (E-\varepsilon,E+\varepsilon)$, and for any $j\neq k$, one has
    $E_j\notin (E-\eta,E+\eta)$.
    We have
    \begin{equation*}
    \varepsilon^2>\|(A-E)\phi\|^2\ge\sum_{j\neq k}|\langle \phi,\psi_j \rangle|^2(E_j-E)^2
    	\ge \eta^2\sum_{j\neq k}|\langle \phi,\psi_j\rangle |^2,
    \end{equation*}
    and therefore
    \begin{equation*}
	    1-| \langle \phi,\psi_k\rangle|^2=\|\phi-\langle \phi,\psi_k \rangle \psi_k\|^2
	    =\sum_{j\neq k}|\langle \phi,\psi_j\rangle|^2\le \varepsilon^2\eta^{-2}.
    \end{equation*}
    The conclusion now follows from the fact that
    $ \norm{\phi-\psi_k}^2=2(1-\Re \langle \phi,\psi_k\rangle) $.
\end{proof}

We will also use the following well-known result (which could be replaced by considerations about semi-algebraic sets).

\begin{lemma}\label{lem:component-count}
	Let $ \omega\in \T $, $ N_1,N_2\ge 1 $, $ \delta>0 $, and assume that the potential $ V $
	in \cref{eq:1.Sch} is a trigonometric polynomial of degree $ d_0 $. Then the number of connected components
	of
	\begin{equation*}
		\left\{ x\in \T: \left|E_{j_1}^{(N_1)}(x,\omega)-E_{j_2}^{(N_2)}(x,\omega)\right|\le \delta \right\}
	\end{equation*}
	is $ \lesssim N_1N_2d_0^2 $.
\end{lemma}

\begin{proof}
	It can be seen from \cref{eq:Dirichlet-det} that
	\begin{equation*}
		\exp(2\pi i d_0Nx)f_N(x,\omega,E)=P_N(\exp(2\pi i x),E),
	\end{equation*}
	with $ P_N $ being a polynomial of degree $ 2d_0N $.
	Since the eigenvalues are continuous in $ x $, the number of components of the set we are interested in
	is bounded by the number of solutions of the system
	\begin{equation*}
		\begin{cases}
			0=P_{N_1}(z,E)\\
			0=P_{N_2}(z,E\pm \delta)
		\end{cases}.
	\end{equation*}
	The conclusion follows by using B\'ezout's Theorem.
\end{proof}

We are now ready to prove a detailed result on the stability of the finite scale spectra for the almost Mathieu operator. To be more precise, the result only applies to certain spectral segments. However, by \cref{lem:spectrum-criterion} and \cref{lem:initial-spectral-segment} we know that these
are precisely the spectral segments that we need to get control of the full scale spectrum.

\addtocounter{propAlpha}{2}

\begin{propAlphaM}\label{prop:stabilization-AMO}
    Consider the almost Mathieu operator \eqref{eq:3.amopera} with $|\lambda|>1$ and
    $\omega\in \mathbb{T}_{c,a}$. Let
    \begin{equation*}
    	\cB_{N,k,\omega}= \left\{ x\in \T: \min_{|n|\le 2N^{(k)}} \norm{x-\frac{n\omega}{2}}
			<\exp\left(-(N^{(k-2)})^{1/2}\right)  \right\},
    \end{equation*}
    where $ N^{(k)}=N^{2^k} $, $ k\ge 0 $, $ N^{(-1)}=N $.
    Let $ C_0 $ be as in \cref{prop:localization-AMO}.
    There exists $N_0(|\lambda|,c,a)$
    such that the following statement holds for any $N\ge N_0$, $ k_0\ge 1 $.
	If there exist $ j_0 $  such that
	\begin{equation*}
		\left| \psi_{j_0}^{[-N,N]}(x,\omega;\pm N) \right|\le \exp(-c_0N),
	\end{equation*}
	for some constant $ c_0<1 $ and
	for any $ x $ in an interval $ I $, $ |I|\le \exp(-(\log 2N^{(k_0-2)})^{C_0}) $,
	then $ I\setminus \bigcup_{k=1}^{k_0}\cB_{N,k,\omega} $ can be partitioned into subintervals
	$ I_m $, $ m\le (N^{(k_0)})^C $, with $ C $ an absolute
	constant, and for each $ I_m $, there exist
	$ N_k=N_k(I_m)\simeq N^{(k)} $ and $ j_k=j_k(I_m) $, $ k=1,\ldots,k_0 $ such that for any
	$ x\in I_m $ and $ k\le k_0-1 $, we
	have
	\EQ{
		\left| E_{j_0}^{[-N,N]}(x,\omega)-E_{j_1}^{[-N_1,N_1]}(x,\omega) \right|
		&\lesssim \exp(-c_0 N),
	\\
		\left| E_{j_{k+1}}^{[-N_{k+1},N_{k+1}]}(x,\omega)-E_{j_k}^{[-N_k,N_k]}(x,\omega) \right|
		&\le \exp(-(N_k\log|\lambda|)/2),
	\\
		\left| \psi_{j_{1}}^{[-N_1,N_1]}(x,\omega;n) \right|
		&\le \exp(-(|n|\log|\lambda|)/2),\ |n|> C(|\lambda|,a,c) N,
	\\
		\left| \psi_{j_{k+1}}^{[-N_{k+1},N_{k+1}]}(x,\omega;n) \right|
		&\le \exp(-(|n|\log|\lambda|)/2),\ |n|> N_k,
	\\
		\left| \psi_{j_{k+1}}^{[-N_{k+1},N_{k+1}]}(x,\omega;n)
			-\psi_{j_k}^{[-N_k,N_k]}(x,\omega;n) \right|
		&\le \exp(-(N_k\log|\lambda|)/2),\ |n|\le N_k.
	}
\end{propAlphaM}

\addtocounter{propAlpha}{-1}

\begin{proof}
	Let $ \cB=\bigcup_{k=1}^{k_0} \cB_{N,k,\omega} $. Note that $ I\setminus \cB $ has
	$ \lesssim (N^{(k_0)})^2 $ components.
	Let $ x_0 $ be the midpoint of $ I $ and let $ E_0=E_{j_0}^{[-N,N]}(x_0,\omega) $. Let $ N'=3N $.
	We choose $ N_i(x_0,E_0,N')\simeq N^{2i} $, $ i=1,2 $ as in \cref{prop:localization-AMO}.
	Since
	$$ s+[-N',N']\supset [-N,N] \text{\ \  for any\ \ } |s|\le [N'/2], $$ it follows from
	\cref{lem:eigenvalue-stabilization} that
	\begin{equation}\label{eq:N'}
		\max_{|s|\le [N'/2]}\dist(E_{j_0}^{[-N,N]}(x,\omega),\spec H_{s+[-N',N']}(x,\omega))
		\le 2 \exp(-c_0N).
	\end{equation}
	\cref{lem:eigenvalue-stabilization} also implies that there exists $ j_1(x) $ such that
	\begin{equation}\label{eq:j_0-j_1}
		\left| E_{j_0}^{[-N,N]}(x,\omega)-E_{j_1}^{[-N_1,N_1]}(x,\omega) \right|\le 2\exp(-c_0N).
	\end{equation}
	It follows from \cref{lem:component-count} that we can partition $ I $ into fewer than $ N_1^C $
	subintervals, with $ C $ an absolute constant, such that we can choose $ j_1(x) $ to be constant on each
	of the
	subintervals. Let $ I^{(1)}_m $, $ m\lesssim (N^{(k_0)})^C $, be the intervals of the partition induced
	on $ I\setminus \cB $.
	Because of \cref{eq:N'}, \cref{eq:j_0-j_1} we have
	\begin{equation*}
		\max_{|s|\le [N'/2]}\dist(E_{j_1}^{[-N_1,N_1]}(x,\omega),\spec H_{s+[-N',N']}(x,\omega))
		\le \exp(-(N')^{1/2+}),
	\end{equation*}	
	so we can apply \cref{prop:localization-AMO} to get
	\begin{equation*}
		\left| \psi_{j_1}^{[-N_1,N_1]}(x,\omega;n) \right|
		\le \exp(-|n|\log|\lambda|+C'N')
		\le \exp(-|n|\log|\lambda|/2),\ |n|>CN
	\end{equation*}
	for all $ x\in I\setminus \cB $. Note that we have
	\begin{equation*}
		\left| E_{j_1}^{[-N_1,N_1]}(x,\omega)-E_0 \right|\le \exp(-(\log N')^{C_0}),
	\end{equation*}
	because of \cref{eq:j_0-j_1}, \cref{eq:Hx-vs-Hx0}, and our assumption on the length of $ I $.
	
	Now \cref{lem:eigenvalue-stabilization} yields the existence of $ j_2(x) $ for each
	$ x\in I\setminus\cB $ such that
	\begin{equation*}
		\left| E_{j_1}^{[-N_1,N_1]}(x,\omega)-E_{j_2}^{[-N_2,N_2]}(x,\omega) \right|
		\le 2\exp(-N_1\log|\lambda|+CN')\le \exp(-N_1\log|\lambda|/2).
	\end{equation*}
	Using \cref{lem:component-count} we obtain a refined partition $ I^{(2)}_m $,
	$ m\lesssim (N^{(k_0)})^C $, of $ I\setminus \cB $, that contains at most $ \lesssim N_2^C $ more intervals
	than the previous one, and such that the choice of $ j_2 $ is constant on each $ I_m^{(2)} $.
	Again we have
	\begin{equation*}
		\max_{|s|\le [N'/2]}\dist(E_{j_2}^{[-N_2,N_2]}(x,\omega),\spec H_{s+[-N',N']}(x,\omega))
		\le \exp(-(N')^{1/2+}),
	\end{equation*}
	and
	\begin{equation*}
			\left| E_{j_2}^{[-N_2,N_2]}(x,\omega)-E_0 \right|\le \exp(-(\log N')^{C_0}),
	\end{equation*}
	so we can apply \cref{prop:localization-AMO} to get the localization estimate for
	$ \psi_{j_2}^{[-N_2,N_2]}(x,\omega) $. Furthermore, we can
	also apply \cref{prop:separation-AMO} together with \cref{lem:eigenvector-stabilization} to
	get
	\begin{multline*}
		\norm{\psi_{j_2}^{[-N_2,N_2]}(x,\omega)-\tilde \psi_{j_1}^{[-N_1,N_1]}(x,\omega)}\\
			\lesssim \exp(-N_1\log|\lambda|+CN')\exp(C'N')\le \exp(-(N_1\log|\lambda|)/2),
	\end{multline*}
	where $ \tilde \psi_{j_1}^{[-N_1,N_1]}(x,\omega) $ is the extension, with zero entries, of
	$ \psi_{j_1}^{[-N_1,N_1]}(x,\omega) $ to $ [-N_2,N_2] $.
	
	The conclusion follows through iteration. For the sake of clarity we set up the next step. Let
	$ x_{1} $ be the midpoint of $ I^{(2)}_m $ and let $ E_1=E_{j_1}^{[-N_1,N_1]}(x_1,\omega) $.
	Let $ N_1'=3N_1 $. We have
	\begin{equation*}
		\max_{|s|\le [N_1'/2]} \dist(E_{j_1}^{[-N_1,N_1]}(x,\omega),\spec H_{s+[-N_1'+N_1']}(x,\omega))
		\le \exp(-(N_1')^{1/2+}),
	\end{equation*}
	and we choose $ N_3(x_1,E_1,N_1')\simeq N^{(3)} $. As before we obtain a refined partition
	$ I^{(3)}_m $, $ m\lesssim (N^{(k_0)})^C $, and for each interval, there exists $ j_3 $ such that
	\begin{equation*}
		\left| E_{j_1}^{[-N_2,N_2]}(x,\omega)-E_{j_3}^{[-N_3,N_3]}(x,\omega) \right|
		\le \exp(-N_2\log|\lambda|/2)
	\end{equation*}
	for all $ x $ in the interval. As before we can apply \cref{prop:localization-AMO},
	\cref{prop:separation-AMO}, and \cref{lem:eigenvector-stabilization} to deduce the desired estimates on
	the eigenvectors.
\end{proof}

The result for general analytic potentials is analogous. Its proof is similar to that of \cref{prop:stabilization-AMO}, but we have to approximate the potential $ V $ by trigonometric polynomials in order to be able to use \cref{lem:component-count}.

\begin{propAlpha}\label{prop:stabilization}
    Consider the Schr\"odinger operator \cref{eq:1.Sch} with real analytic $ V $. Assume that
    $ L(\omega,E)\ge \gamma>0 $ for any $ \omega $ and any $ E\in(E',E'') $. Let
    $ \cB_{\ell,N,\omega} $, $ \Omega_{\ell,N} $ be as in \cref{prop:localization} and
    \begin{equation*}
    	\cB_{N,k,\omega}=\cB_{\ell_k,N_k,\omega},\quad \Omega_{N,k}=\Omega_{\ell_k,N_k},
    \end{equation*}
    with $ N_k=[\exp(N^{1/10})]^{2^{k-1}} $, $ \ell_k=3[\log N_k]^{10} $.
    There exists $N_0(V,c,a,\gamma)$
    such that the following statement holds for any $N\ge N_0$, $ k_0\ge 1 $ and
    $ \omega\in \T_{c,a}\setminus \bigcup_{k=1}^{k_0}\Omega_{N,k} $.
	If there exists $ j_0 $ such that
	\begin{equation*}
		\left| \psi_{j_0}^{[-N,N]}(x,\omega;\pm N) \right|\le \exp(-c_0N),
	\end{equation*}
	for some constant $ c_0<1 $ and for any $ x $ in an interval $ I $, and $ E_{j_0}^{[-N,N]}(I,\omega)\subset(E',E'') $,
	then $ I\setminus \bigcup_{k=1}^{k_0}\cB_{N,k,\omega} $ can be partitioned into subintervals
	$ I_m $, $ m\le N_{k_0}^C $, with $ C $ an absolute
	constant, and for each $ I_m $, there exist
	$ N_k=N_k(I_m)\simeq N^{(k)} $ and $ j_k=j_k(I_m) $, $ k=1,\ldots,k_0 $ such that for any
	$ x\in I_m $ and $ k\le k_0-1 $, we
	have
	\EQ{
		  \left| E_{j_0}^{[-N,N]}(x,\omega)-E_{j_1}^{[-N_1,N_1]}(x,\omega) \right|
		&\lesssim \exp(-c_0 N),
	\\
		\left| E_{j_{k+1}}^{[-N_{k+1},N_{k+1}]}(x,\omega)-E_{j_k}^{[-N_k,N_k]}(x,\omega) \right|
		&\le \exp(-(N_k\gamma)/2),
	\\
	 	\left| \psi_{j_{1}}^{[-N_1,N_1]}(x,\omega;n) \right|
		&\le \exp(-(|n|\gamma)/2),\ |n|> \exp\left((\log N)^{C(V,c,a,\gamma)}\right),
	\\
	 	\left| \psi_{j_{k+1}}^{[-N_{k+1},N_{k+1}]}(x,\omega;n) \right|
		&\le \exp(-(|n|\gamma)/2),\ |n|> N_k,
	\\
		\left| \psi_{j_{k+1}}^{[-N_{k+1},N_{k+1}]}(x,\omega;n)
			-\psi_{j_k}^{[-N_k,N_k]}(x,\omega;n) \right|
		&\le \exp(-(N_k\gamma)/2),\ |n|\le N_k.
 }
\end{propAlpha}

\begin{proof}
	Let $ \cB=\bigcup_{k=1}^{k_0} \cB_{N,k,\omega}$. It follows from \cref{eq:complexity-estimates} that
	$ I\setminus \cB $ has $ \lesssim (N_{k_0})^C $ intervals. Note that $ \ell_1\simeq 3N $, so
	$ s+[-\ell_1,\ell_1]\supset[-N,N] $ for any $ |s|\le[\ell_1/2] $, and therefore by
	\cref{lem:eigenvalue-stabilization},
	\begin{equation*}
		\max_{|s|\le[\ell_1/2]}\dist(E_{j_0}^{[-N,N]}(x,\omega),\spec H_{s+[-\ell_1,\ell_1]}(x,\omega))
		\le 2\exp(-c_0 N).
	\end{equation*}
	\cref{lem:eigenvalue-stabilization} also implies that there exists $ j_1(x) $ such that
	\begin{equation*}
		\left| E_{j_0}^{[-N,N]}(x,\omega)-E_{j_1}^{[-N_1,N_1]}(x,\omega) \right|\le 2\exp(-c_0N).
	\end{equation*}
	Choose $ \tilde V $ as in \cref{eq:V_K} with $ K=CN_1 $ such that, by \cref{eq:trig-approximation},
	we have
	\begin{equation*}
		\left| \tilde E_{j_0}^{[-N,N]}(x,\omega)-\tilde E_{j_1}^{[-N_1,N_1]}(x,\omega) \right|
		\le 3\exp(-c_0N).
	\end{equation*}
	By \cref{lem:component-count} we can partition $ I $ into fewer than $ N_1^C $ subintervals such that
	 $ j_1(x) $ can be kept constant on each of the subintervals. Using \cref{eq:trig-approximation}
	again, on these intervals we have
	\begin{equation*}
		\left| E_{j_0}^{[-N,N]}(x,\omega)-E_{j_1}^{[-N_1,N_1]}(x,\omega) \right|\le 4\exp(-c_0N),
	\end{equation*}	
	and one can proceed as in the proof of \cref{prop:stabilization-AMO}. We just note that the choice of
	$ \ell_k $ and $ N_k $ is such that the separation obtained from \cref{prop:separation} is by
	$ \exp(-N_k^\epsilon)>\exp(-N_{k_1}^{1/2}) $. This is crucial for obtaining the desired estimates from
	\cref{lem:eigenvector-stabilization}.
\end{proof}

Finally let us note that our main results also follow from the results on stabilization (though for general potentials the
result is weaker because we have to remove a measure zero set of bad frequencies). This is simply because we
can establish the following two analogues of \cref{lem:finite-segment-to-full-spectrum-Bourgain}. Their proofs mirror
that of \cref{lem:finite-segment-to-full-spectrum-Bourgain}. For the convenience of the reader, we include the proof for
the almost Mathieu case.

\begin{proposition}\label{prop:finite-segment-to-full-spectrum-AMO}
	Consider the almost Mathieu operator \eqref{eq:3.amopera} with $|\lambda|>1$ and
    $\omega\in \mathbb{T}_{c,a}$.
    There exists $N_0(|\lambda|,c,a)$ such that the following statement holds for any $N\ge N_0$.
	If there exist $ j_0 $ and an interval $ I\subset \T $ such that
	\begin{equation*}
		\left| \psi_{j_0}^{[-N,N]}(x,\omega;\pm N) \right|\le \exp(-c_0N), \ x\in I
	\end{equation*}
	for some constant $ c_0<1 $,
	then
	\begin{equation*}
		\mes(E_{j_0}^{[-N,N]}(I,\omega)\setminus \cS_\omega)\le \exp(-c_1N^{1/2}),
	\end{equation*}	
	with $ c_1 $ an absolute constant.
\end{proposition}

\begin{proof}
	Let $ \cB_{N,k,\omega} $ be as in \cref{prop:stabilization-AMO}. Let $ C_0 $ be as in
	\cref{prop:localization-AMO}. Partition $ I $ into intervals $ I^{(0)}_m $,
	$ m\lesssim \exp((\log 2N)^{C_0}) $, such that $ |I^{(0)}_m|\le \exp(-(\log 2N)^{C_0}) $.
	Let $ I^{(1)}_m $, $$ m\lesssim N^C\exp((\log 2N)^{C_0})\le \exp((\log N)^C) $$ be
	the partition of $ I\setminus \cB_{N,1,\omega} $ obtained by applying \cref{prop:stabilization-AMO}
	with $ k_0=1 $ on each $ I^{(0)}_m $. Since on each $ I^{(1)}_m $ we have
	\begin{equation*}
		\left| E_{j_0}^{[-N,N]}(x,\omega)-E_{j_1}^{[-N_1,N_1]}(x,\omega) \right|
		\le \exp(-cN)
	\end{equation*}
	 with $ N_1=N_1(I^{(1)}_m)\simeq N^2 $, $ j_1=j_1(I^{(1)}_m) $, it follows by the continuity of the
	 parametrization of the eigenvalues that
	 \begin{equation*}
	 	\mes \left(E_{j_0}^{[-N,N]}(I^{(1)}_m,\omega) \ominus
			E_{j_1}^{[-N_1,N_1]}(I^{(1)}_m,\omega)\right)\lesssim \exp(-cN),
	 \end{equation*}
	 where $ \ominus $ denotes the symmetric difference of two sets. From \cref{eq:Hx-vs-Hx0} it follows
	 that
	 \begin{equation*}
	 	\mes \left( E_{j_0}^{[-N,N]}(\cB_{N,1,\omega},\omega)\right)\le \exp(-cN^{1/2}).
	 \end{equation*}
	 Let
	 \begin{equation*}
	 	\cE_{N,1,\omega}=E_{j_0}^{[-N,N]}(\cB_{N,1,\omega},\omega)\cup
		\left(\bigcup_{m}E_{j_0}^{[-N,N]}(I^{(1)}_m,\omega) \ominus
			E_{j_1}^{[-N_1,N_1]}(I^{(1)}_m,\omega)\right).
	 \end{equation*}
	 We clearly have $ \mes(\cE_{N,1,\omega})\le \exp(-cN^{1/2}) $.
	
	 Note that \cref{lem:eigenvalue-stabilization} implies that
	 \begin{equation*}
	 	\dist(E,\cS_\omega)\lesssim \exp(-c_0N), \ E\in E_{j_0}^{[-N,N]}(I,\omega).
	 \end{equation*}
	 Since any $ E\in E_{j_0}^{[-N,N]}(I,\omega)\setminus \cE_{N,1,\omega} $ also belongs to some
	 $ E_{j_1}^{[-N_1,N_1]}(I^{(1)}_m,\omega) $, it follows from \cref{prop:stabilization-AMO} and
	 \cref{lem:eigenvalue-stabilization} that
	 \begin{equation*}
	 	\dist(E,\cS_\omega)\le \exp(-cN_1), \ E\in E_{j_0}^{[-N,N]}(I,\omega)\setminus \cE_{N,1,\omega}.
	 \end{equation*}
	
	 By applying \cref{prop:stabilization-AMO} repeatedly, we obtain sets
	 \begin{equation*}
	 	\cE_{N,k,\omega}=E_{j_{0}}^{[-N,N]}(\cB_{N,k,\omega},\omega)\cup
		\left(\bigcup_{m}E_{j_{k-1}}^{[-N_{k-1},N_{k-1}]}(I^{(k)}_m,\omega) \ominus
			E_{j_k}^{[-N_k,N_k]}(I^{(k)}_m,\omega)\right)
	 \end{equation*}
	 such that $ \mes(\cE_{N,k,\omega})\le \exp(-c(N^{(k-2)})^{1/2}) $ (recall that $ N^{(k)}=N^{2^k} $) and
	 \begin{equation*}
	 	\dist(E,\cS_\omega)\le \exp(-cN^{(k)}),
		\ E\in E_{j_0}^{[-N,N]}(I,\omega)\setminus \bigcup_{\ell=1}^k \cE_{N,\ell,\omega}.
	 \end{equation*}
	 Finally, we note that
	 \begin{equation*}
	 	E_{j_0}^{[-N,N]}(I,\omega)\setminus \cS_\omega\subset \bigcup_k \cE_{N,k,\omega},
	 \end{equation*}
	 and we are done.
\end{proof}

\begin{proposition}\label{prop:finite-segment-to-full-spectrum}
    Consider the Schr\"odinger operator \cref{eq:1.Sch} with real analytic $ V $. Assume that
    $ L(\omega,E)\ge \gamma>0 $ for any $ \omega $ and any $ E\in(E',E'') $. Let $ \Omega_{N,k} $ be
    as in \cref{prop:stabilization}.
    There exists $N_0(V,c,a,\gamma)$ such that the following statement holds for any $N\ge N_0$ and any
    $ \omega\in\T_{c,a}\setminus \bigcup_{k\ge 1}\Omega_{N,k} $.
	If there exist $ j_0 $ and an interval $ I\subset \T $ such that
	\begin{equation*}
		\left| \psi_{j_0}^{[-N,N]}(x,\omega;\pm N) \right|\le \exp(-c_0N), \ x\in I
	\end{equation*}
	for some constant $ c_0<1 $,
	and $ E_{j_0}^{[-N,N]}(I,\omega)\subset(E',E'') $,
	then
	\begin{equation*}
		\mes(E_{j_0}^{[-N,N]}(I,\omega)\setminus \cS_\omega)\le \exp(-c_1N^{1/2}),
	\end{equation*}	
	with $ c_1 $ an absolute constant.
\end{proposition}

\bibliographystyle{alpha}
\def\cprime{$'$}

\end{document}